\documentclass{article}
\usepackage[utf8]{inputenc} \usepackage[T1]{fontenc}    \usepackage[english]{babel}

\usepackage{amsfonts}
\usepackage{nicefrac}       \usepackage{microtype}      \usepackage{lmodern}
\usepackage{amssymb,amsmath,amsthm}
\usepackage{stmaryrd}
\usepackage{bbm,bm}
\usepackage{latexsym}
\usepackage{xcolor}

\usepackage{enumerate}
\usepackage{verbatim}
\usepackage{booktabs}       

\usepackage{url}            \usepackage[colorlinks=true]{hyperref}
\usepackage[numbers,sort&compress,square,comma]{natbib}
\usepackage[capitalise,noabbrev]{cleveref}

\usepackage{parskip}
\usepackage{geometry}
\usepackage[short]{optidef}
\usepackage{algorithm}
\usepackage{authblk}

\usepackage{thmtools}

\declaretheorem{theorem}
\declaretheorem{corollary}
\declaretheorem{lemma}
\declaretheorem{proposition}
\declaretheorem{observation}

\declaretheoremstyle[qed=$\square$]{definitionwithend}
\declaretheorem[style=definitionwithend]{definition}
\declaretheorem[style=definitionwithend]{assumption}
\declaretheorem[style=definitionwithend]{example}
\declaretheorem[style=definitionwithend]{remark}

\crefname{fact}{Fact}{Facts}
\crefname{algorithm}{Algorithm}{Algorithms}
\crefname{assumption}{Assumption}{Assumptions}

\definecolor{gold}{rgb}{0.85,0.65,0}

\newcommand{\ceil}[1]{\ensuremath{\left\lceil #1 \right\rceil}}
\newcommand{\abs}[1]{\ensuremath{\left\lvert #1 \right\rvert}}

\newcommand{\by}{\times}
 
\newcommand{\norm}[1]{\ensuremath{\left\lVert #1 \right\rVert}}
\newcommand{\ip}[1]{\ensuremath{\left\langle #1 \right\rangle}}
\newcommand{\grad}{\ensuremath{\nabla}}

\newcommand{\set}[1]{\left\{#1\right\}}

\def\R{{\mathbb{R}}}
\def\S{{\mathbb{S}}}

\def\cI{{\cal I}}

\DeclareMathOperator{\Opt}{Opt}

\DeclareMathOperator*{\argmax}{arg\,max}

\DeclareMathOperator{\spann}{span}
\DeclareMathOperator{\inter}{int}

 \usepackage{soul}

 \usepackage{letltxmacro}
\LetLtxMacro\orgvdots\vdots
\LetLtxMacro\orgddots\ddots

\makeatletter
\DeclareRobustCommand\vdots{\mathpalette\@vdots{}}
\newcommand*{\@vdots}[2]{\sbox0{$#1\cdotp\cdotp\cdotp\m@th$}\sbox2{$#1.\m@th$}\vbox{\dimen@=\wd0 \advance\dimen@ -3\ht2 \kern.5\dimen@
\dimen@=\wd2 \advance\dimen@ -\ht2 \dimen2=\wd0 \advance\dimen2 -\dimen@
    \vbox to \dimen2{\offinterlineskip
      \copy2 \vfill\copy2 \vfill\copy2 }}}
\DeclareRobustCommand\ddots{\mathinner{\mathpalette\@ddots{}\mkern\thinmuskip
  }}
\newcommand*{\@ddots}[2]{\sbox0{$#1\cdotp\cdotp\cdotp\m@th$}\sbox2{$#1.\m@th$}\vbox{\dimen@=\wd0 \advance\dimen@ -3\ht2 \kern.5\dimen@
\dimen@=\wd2 \advance\dimen@ -\ht2 \dimen2=\wd0 \advance\dimen2 -\dimen@
    \vbox to \dimen2{\offinterlineskip
      \hbox{$#1\mathpunct{.}\m@th$}\vfill
      \hbox{$#1\mathpunct{\kern\wd2}\mathpunct{.}\m@th$}\vfill
      \hbox{$#1\mathpunct{\kern\wd2}\mathpunct{\kern\wd2}\mathpunct{.}\m@th$}}}}

\DeclareRobustCommand\bddots{\mathinner{\mathpalette\@bddots{}\mkern\thinmuskip
  }}
\newcommand*{\@bddots}[2]{\sbox0{$#1\cdotp\cdotp\cdotp\m@th$}\sbox2{$#1.\m@th$}\vbox{\dimen@=\wd0 \advance\dimen@ -3\ht2 \kern.5\dimen@
\dimen@=\wd2 \advance\dimen@ -\ht2 \dimen2=\wd0 \advance\dimen2 -\dimen@
    \vbox to \dimen2{\offinterlineskip
      \hbox{$#1\mathpunct{\kern\wd2}\mathpunct{\kern\wd2}\mathpunct{.}\m@th$}\vfill
      \hbox{$#1\mathpunct{\kern\wd2}\mathpunct{.}\m@th$}\vfill
      \hbox{$#1\mathpunct{.}\m@th$}}}}
\makeatother 
\usepackage{enumitem}

\usepackage{thm-restate}

\PassOptionsToPackage{numbers, compress}{natbib}

\newcommand{\ije}[2]{#2}

\begin{document}

\title{Implicit regularity and linear convergence rates for the generalized trust-region subproblem}
\author[1]{Alex L.\ Wang}
\author[2]{Yunlei Lu}
\author[1]{Fatma K{\i}l{\i}n\c{c}-Karzan}
\affil[1]{Carnegie Mellon University, Pittsburgh, PA, 15213, USA.}
\affil[2]{Peking University, Beijing, China, 100871}

\date{\today}

\maketitle
\begin{abstract}
In this paper we develop efficient first-order algorithms for the \emph{generalized trust-region subproblem} (GTRS), which has applications in signal processing, compressed sensing, and engineering.
Although the GTRS, as stated, is nonlinear and nonconvex, it is well-known that objective value exactness holds for its SDP relaxation under a Slater condition.
While polynomial-time SDP-based algorithms exist for the GTRS, their relatively large computational complexity has motivated and spurred the development of custom approaches for solving the GTRS.
In particular, recent work in this direction has developed first-order methods for the GTRS whose running times are linear in the sparsity (the number of nonzero entries) of the input data.
In contrast to these algorithms, in this paper we develop algorithms for computing $\epsilon$-approximate solutions to the GTRS whose running times are linear in both the input sparsity \emph{and} the precision $\log(1/\epsilon)$ whenever a regularity parameter is positive.
We complement our theoretical guarantees with numerical experiments comparing our approach against algorithms from the literature. 
Our numerical experiments highlight that our new algorithms significantly outperform prior state-of-the-art algorithms on sparse large-scale instances. \end{abstract}

\section{Introduction}\label{sec:intro}
In this paper we develop efficient first-order algorithms for the \emph{generalized trust-region subproblem} (GTRS). Recall the GTRS,
\begin{align*}
\Opt\coloneqq \inf_{x\in\R^n}\set{q_0(x):\, q_1(x)\leq 0},
\end{align*}
where $q_0(x)$ and $q_1(x)$ are quadratic functions in $x\in\R^n$. We will assume that for each $i\in\set{0,1}$, the quadratic function $q_i(x)$ is given by $q_i(x) = x^\top A_i x + 2b_i^\top x +c_i$ for $A_i\in\S^n$, $b_i\in\R^n$ and $c_i\in\R$.

This problem generalizes the classical \emph{trust-region subproblem} (TRS) where the general quadratic constraint $q_1(x)\leq 0$ is replaced with the unit ball constraint $\norm{x}^2\leq 1$.
The TRS finds applications, for example, in robust optimization~\cite{benTal2014hidden,hoNguyen2017second} and combinatorial optimization~\cite{pardalos1991algorithms,karmarkar1991interior}.
The TRS is additionally foundational in the area of nonlinear programming. Indeed, iterative algorithms based on the TRS (known sometimes as trust-region methods)~\cite{conn2000trust} are among the most empirically successful techniques for general nonlinear programs.

Generalizing the TRS, the GTRS has applications in signal processing, compressed sensing, and engineering (see~\cite{wang2020generalized} and references therein). The problem of minimizing a quartic of the form $q(x,p(x))$, where $q:\R^{n+1}\to\R$ and $p:\R^n\to\R$ are both quadratic, can be cast in the equality-constrained variant of the GTRS. This approach has been used to address source localization~\cite{hmam2010quadratic} as well as the double-well potential functions~\cite{feng2012duality}.
More broadly, iterative ADMM-based algorithms for general QCQPs using the GTRS as a subprocedure have shown exceptional numerical performance~\cite{huang2016consensus} and outperform previous state-of-the-art approaches on a number of real world problems (e.g., multicast beamforming and phase retrieval). This application of the GTRS as a subprocedure within an iterative solver parallels the use of the TRS within trust-region methods.

Although the GTRS, as stated, is nonlinear and nonconvex, it is well-known that objective value exactness holds for its SDP relaxation under a Slater condition~\cite{fradkov1979s-procedure,polik2007survey}. Thus, unlike general QCQPs which are NP-hard, the GTRS can be solved in polynomial time via SDP-based algorithms. Nevertheless, the relatively large computational complexity of SDP-based approaches has motivated and spurred the development of alternative custom approaches for solving the GTRS.
We restrict our discussion below to \emph{recent} trends in GTRS algorithms and discuss \emph{earlier} work~\cite{more1993generalizations,more1983computing,stern1995indefinite} where appropriate in the main body.

One line of proposed algorithms for the GTRS assumes simultaneous diagonalizability (SD) of $A_0$ and $A_1$.
It is well-known that SD holds under minor conditions---for example, if there exists a positive definite matrix in $\spann\set{A_0,A_1}$ (see~\cite{wang2021new} for additional variants of this result).
Ben-Tal and Teboulle~\cite{benTal1996hidden} exploit the SD condition to provide a reformulation of the interval-constrained GTRS as a convex minimization problem with linear constraints. More recently, under the SD condition, Ben-Tal and den Hertog~\cite{benTal2014hidden} provide a second-order cone program~(SOCP) reformulation of the GTRS in a lifted space. This SOCP reformulation was generalized beyond the GTRS in \cite{locatelli2015some}. 
Under the SD condition, a number of papers~\cite{salahi2016trust,fallahi2018minimizing} exploit the resulting problem structure of the primal or the dual formulation to derive solution procedures for the GTRS and interval-constrained GTRS.
Generalizing~\cite{benTal2014hidden},
Jiang et al.~\cite{jiang2018socp} provide an SOCP reformulation for the GTRS in a lifted space whenever the problem has a finite optimal value even when the SD condition fails.
Unfortunately, the algorithms in this line often assume implicitly that $A_0$ and $A_1$ are already diagonal or that a simultaneously-diagonalizing basis can be computed. In practice, however, computing such a basis requires a full eigen-decomposition and can be prohibitively expensive for large-scale instances.

A second line of research on the GTRS explores the connection between the GTRS and generalized eigenvalues of the matrix pencil $A_0 + \gamma A_1$.
Pong and Wolkowicz~\cite{PongWolkowicz2014} propose a generalized-eigenvalue-based algorithm which exploits the structure of optimal GTRS solutions, albeit without an explicit running time analysis.
Adachi and Nakatsukasa~\cite{adachi2019eigenvalue} present another approach for solving the GTRS based on computing the minimum generalized eigenvalue (and corresponding eigenvector) of an associated \emph{indefinite} $(2n+1)\times (2n+1)$ matrix pencil.
Unfortunately, this approach suffers from the significant cost of  computing a minimum generalized eigenvalue of an indefinite matrix pencil.
Empirically, the complexity of this approach scales as $O(n^2)$ even for sparse instances of the GTRS with $O(n)$ nonzero entries in $A_0$ and $A_1$ (see \cite[Section 4]{adachi2019eigenvalue}).
Jiang and Li~\cite{jiang2019novel} reformulate the GTRS as the problem of minimizing the maximum of two convex  quadratic functions in the original space.
This reformulation is constructed from a pair of generalized eigenvalues related to the matrix pencil $A_0 + \gamma A_1$. They then suggest a saddle-point-based first-order algorithm to solve this reformulation within an $\epsilon$ additive error in $O(1/\epsilon)$ time. 
These approaches are based on the assumption that the generalized eigenvalues are given or can be computed exactly, and offer no theoretical guarantees when only approximate generalized eigenvalue computations are available (as is the case in practice; see also the discussion in Section 4.1 in \cite{jiang2020linear}). 
Despite this, the numerical experiments in~\cite{jiang2019novel,adachi2019eigenvalue,PongWolkowicz2014} suggest that algorithms motivated by these ideas perform well even using only approximate generalized eigenvalue computations.

In contrast to these papers, recent work~\cite{jiang2020linear,wang2020generalized} offers provably linear-time (in terms of the number of nonzero entries in the input data) algorithms for the GTRS using only approximate eigenvalue procedures.
Jiang and Li~\cite{jiang2020linear} extend ideas developed in~\cite{hazan2016linear} for solving the TRS to derive an algorithm for solving the GTRS up to an $\epsilon$ additive error with high probability. This approach differs from the earlier literature in that it does not rely on the computation of a simultaneously-diagonalizing basis or exact generalized eigenvalues. The complexity of this approach is
\begin{align*}
\tilde O\left(\frac{N}{\sqrt{\epsilon}}\log\left(\frac{n}{p}\right)\log\left(\frac{1}{\epsilon}\right)^2\right),
\end{align*}
where $N$ is the number of nonzero entries in $A_0$ and $A_1$, $\epsilon$ is the additive error, and $p$ is the failure probability.
Here, we have elided quantities related to the condition number of the GTRS.
Wang and K{\i}l{\i}n\c{c}-Karzan~\cite{wang2020generalized} reexamine the convex quadratic reformulation idea of \cite{jiang2019novel} and show formally that by approximating the generalized eigenvalues sufficiently well, the perturbed convex reformulation is within a small additive error of the true convex reformulation. Moreover, they establish that the resulting convex reformulation can be solved via Nesterov's accelerated gradient descent method~\cite[Section 2.3.3]{nesterov2018lectures} for smooth minimax problems to achieve an overall run time guarantee of
\begin{align*}
\tilde O\left(\frac{N}{\sqrt{\epsilon}}\log\left(\frac{n}{p}\right)
\log\left(\frac{1}{\epsilon}\right)\right).
\end{align*}

A parallel line of work~\cite{hoNguyen2017second,fortin2004trust,gould1999solving,more1983computing,hazan2016linear,carmon2018analysis} has developed custom first-order methods for the trust-region subproblem.
Most relatedly, Carmon and Duchi~\cite{carmon2018analysis} recently showed that a Krylov-based first-order method
can achieve a convergence rate for the TRS that is linear in both $N$ \emph{and} the precision $\log(1/\epsilon)$ whenever a regularity parameter, $\mu^*$, is positive. This contrasts with previous algorithms for the TRS whose guarantees scaled as $\approx 1/\sqrt{\epsilon}$.

In this paper, we introduce and analyze a \emph{new} algorithm for computing an $\epsilon$-approximate solution to the GTRS whose running time is linear in both $N$ and the precision $\log(1/\epsilon)$ whenever $\mu^*$ is positive.
To be concrete, an $\epsilon$-approximate solution is defined below.
\begin{definition}
We say $x\in\R^n$ is an \emph{$\epsilon$-approximate solution} to \eqref{eq:gtrs} if
\begin{gather*}
q_0(x)\leq \Opt + \epsilon\quad\text{and}\quad q_1(x) \leq \epsilon.\qedhere
\end{gather*}
\end{definition}
Despite similar convergence guarantees, our approach for solving the GTRS does not share many algorithmic similarities with the approach of Carmon and Duchi~\cite{carmon2018analysis} for the TRS.

\subsection{Overview and outline of paper}

A summary of our contributions, along with an outline of the remainder of the paper, is as follows:
\begin{itemize}
	\item In \cref{sec:regularity}, we recall definitions and results related to the Lagrangian dual of the GTRS and define our notion of regularity. Specifically, we recall definitions and results in the literature~\cite{more1993generalizations,more1983computing,adachi2019eigenvalue,feng2012duality} regarding the dual function $\mathbf{d}(\gamma)$ and its derivative $\nu(\gamma)$. We then define a \emph{regularity} parameter $\mu^*$, which will play the role of strong convexity in our algorithms. We close with a key lemma (\cref{lem:gtrs_stability}) that underpins the algorithms developed in this paper. Intuitively, \cref{lem:gtrs_stability} says that when $\mu^*$ is positive, the unique optimizer of the GTRS is stable---an $\Omega(\mu^*)$-strongly convex reformulation of the GTRS, whose unique optimizer coincides with the GTRS optimizer, can be built using \emph{inexact} estimates of the dual optimizer $\gamma^*$.
\item In \cref{sec:constructing_sc_reform}, we describe and analyze an approach for computing an $\epsilon$-approximate optimizer of a nonconvex-nonconvex GTRS instance based on \cref{lem:gtrs_stability}. 
	Our approach consists of two algorithms, \texttt{ConstructReform} and \texttt{SolveRegular}.
	The first algorithm uses inexact estimates of $\nu(\gamma)$ to binary search for an inexact estimate of $\gamma^*$. \texttt{ConstructReform} will either return an \emph{exact} $\Omega(\mu^*)$-strongly convex reformulation of the GTRS or an $\epsilon$-approximate optimizer of the GTRS. In the former case, we may then apply \texttt{SolveRegular} to compute an $\epsilon$-approximate optimizer.
	In the latter case, \texttt{ConstructReform} will additionally \emph{attempt} to certify that $\mu^*= O(\epsilon)$ so that building an $\Omega(\mu^*)$-strongly convex reformulation may be undesirable.
Together, these two algorithms achieve the following linear convergence rate (i.e., scaling as $\log(1/\epsilon)$) for the GTRS:
	\begin{align*}
	\tilde O\left(\frac{N}{\sqrt{\phi}}\log\left(\frac{1}{\phi}\right)\log\left(\frac{n}{p}\right)\log\left(\frac{1}{\epsilon}\right)\right).
	\end{align*}
	Here, $N$ is the number of nonzero entries in $A_0$ and $A_1$ combined,
$\phi$ can be thought of as $\approx \max(\mu^*,\epsilon)$
	(see \cref{sec:constructing_sc_reform} for a formal definition), $p$ is the failure probability, and the $\tilde O$-notation hides $\log\log$-factors.
	This contrasts with previous algorithms~\cite{wang2020generalized,jiang2020linear} that are described as ``linear-time'', referring to the fact that their running times scale linearly in only $N$.
We close this section by examining in further detail the case where \texttt{ConstructReform} returns an $\epsilon$-approximate optimizer but fails to certify that $\mu^* = O(\epsilon)$. Specifically, we show that this edge case can only happen if $\nu(\gamma)$ is ``extremely flat,'' which in turn can only happen if a certain \emph{coherence} parameter is small.
\item In \cref{sec:numerical}, we present numerical experiments comparing the algorithms of \cref{sec:constructing_sc_reform} to other algorithms proposed in the recent literature~\cite{benTal2014hidden,jiang2019novel,adachi2019eigenvalue}. Our numerical experiments corroborate our theoretical understanding of the situation---the algorithms in this paper significantly outperform prior state-of-the-art algorithms on sparse large-scale GTRS instances.
\end{itemize}

\subsection{Notation}
For $x\in\R$ and $y\geq 0$ let 
$[\pm y]\coloneqq [-y, +y]$ and $[x\pm y] \coloneqq [x-y, x+y]$.
We denote the $i$-th unit vector in $\R^n$ by $e_i$. 
Let $\S^n$ denote the $n\by n$ real symmetric matrices.
For $A\in\S^n$ we will write $A\succeq 0$ (resp.\ $A\succ 0$) to denote that $A$ is positive semidefinite (resp.\ positive definite).
For $\gamma\in\R_+$, define $A(\gamma) \coloneqq A_0 + \gamma A_1$, $b(\gamma)\coloneqq b_0 + \gamma b_1$, and $c(\gamma)\coloneqq c_0 + \gamma c_1$.
Let $q(\gamma,x)\coloneqq q_0(x) + \gamma q_1(x)$.
For $A\in\S^n$, let $\norm{A}$ be its spectral norm.
For $b\in\R^n$, let $\norm{b}$ be its Euclidean norm.
For an interval $\Gamma\subseteq\R$, let $\inter(\Gamma)$ denote its interior.
We will use $\tilde O$-notation to hide $\log\log$-factors in our running times.

\section{Implicit Regularity in the GTRS}
\label{sec:regularity}
Recall that the GTRS is the problem of minimizing a quadratic objective function subject to a single quadratic constraint, i.e.,
\begin{align}
\label{eq:gtrs}
\Opt\coloneqq \inf_{x\in\R^n}\set{q_0(x):\, q_1(x) \leq 0},
\end{align}
where for each $i\in\set{0,1}$, we have $q_i(x) = x^\top A_i x + 2b_i^\top x + c_i$ for some $A_i\in\S^n$, $b_i\in\R^n$, and $c_i\in\R$.

We will make the following \emph{blanket} assumption, which is both natural and common in the literature on the GTRS~\cite{wang2020generalized,adachi2019eigenvalue,jiang2016simultaneous,jiang2020linear}. This assumption can be thought of as primal and dual strict feasibility assumptions or a Slater assumption.
\begin{assumption}
\label{as:definiteness}
There exists $\bar x\in\R^n$ such that $q_1(\bar x)<0$ and there exists $\bar \gamma\geq 0$ such that $A(\bar \gamma)\succ 0$.
\end{assumption}
\begin{remark}
\label{rem:definiteness_trs}
Note, for example, that \cref{as:definiteness} holds in the classical TRS setting where $q_1(x) = x^\top x - 1$. Indeed, $q_1(0) < 0$ and $A(\gamma) = A_0 + \gamma I \succ 0$ for all $\gamma$ large enough.\ije{\hfill\proofbox}{}
\end{remark}

The results and definitions will assume only \cref{as:definiteness}. In particular, they can be applied to both the classical TRS setting as well as the nonconvex-nonconvex GTRS setting of \cref{sec:constructing_sc_reform}.

Let $\Gamma\coloneqq \set{\gamma\in\R_+:\, A(\gamma)\succeq 0}$. This is a closed interval as the positive semidefinite cone is closed.
If $\Gamma$ is bounded, let $[\gamma_-,\gamma_+]$ denote its left and right endpoints. Else, let $\gamma_-$ denote its left endpoint and define $\gamma_+\coloneqq + \infty$.
Note that for any $\gamma\in\Gamma$, $q(\gamma,x)$ is a convex function of $x$.
Furthermore, by the existence of $\bar\gamma\geq 0$ such that $A(\bar\gamma)\succ 0$, we have that $0 \leq \gamma_-<\gamma_+$.

\begin{definition}
Let $\mathbf{d}:\R_+\to\set{-\infty}\cup\R$ denote the extended-real-valued function defined by
\begin{align*}
\mathbf{d}(\gamma)&\coloneqq \inf_{x\in\R^n} q(\gamma,x).\qedhere
\end{align*}
\end{definition}

We make the following observations on $\mathbf{d}(\gamma)$.
\begin{observation}
\label{obs:dual_properties}
Suppose \cref{as:definiteness} holds. Then,
\begin{itemize}
	\item The function $\mathbf{d}(\gamma)$ is continuous and concave as it is the infimum of affine functions of $\gamma$.
	\item For $\gamma \in\R_+\setminus\Gamma$, the function $q(\gamma, x)$ is nonconvex in $x$ so that $\mathbf{d}(\gamma) = -\infty$.
	\item As $q_1(\bar x)<0$, we have $\mathbf{d}(\gamma) \leq q(\gamma,\bar x) \to -\infty$ as $\gamma\to\infty$.
\end{itemize}
\end{observation}

We comment on the connection between $\mathbf{d}(\gamma)$, the SDP relaxation of \eqref{eq:gtrs}, and the Lagrangian dual of \eqref{eq:gtrs}.
One consequence of the S-lemma~\cite{fradkov1979s-procedure} is that the GTRS has an exact SDP relaxation. Furthermore, it is well-known that the SDP relaxation of a general quadratically constrained quadratic program is equivalent to its Lagrangian dual~\cite{benTal2001lectures}. We will write this fact in our setting as the following identity (which holds under \cref{as:definiteness}),
\begin{align}
\label{eq:opt_equals_inf_sup}
\Opt = \inf_{x\in\R^n}\sup_{\gamma\in\Gamma} q(\gamma, x).
\end{align}
We provide a short self-contained proof of this fact in~\cref{app:regularity}.
Next, by coercivity~\cite[Proposition VI.2.3]{ekeland1999convex} we have that
\begin{align}
\label{eq:opt_equals_sup_inf}
\Opt = \sup_{\gamma\in\Gamma}\inf_{x\in\R^n}q(\gamma,x) = \sup_{\gamma\in\Gamma}\mathbf{d}(\gamma) = \sup_{\gamma\in\R_+}\mathbf{d}(\gamma).
\end{align}

In words, \eqref{eq:opt_equals_inf_sup} shows that the GTRS can be written as a convex minimization problem. Specifically, we can write $\Opt$ in one of the two following ways, corresponding respectively to the cases $\gamma_+<\infty$ and $\gamma_+=\infty$:
\begin{align}
\label{eq:gtrs_convex_reform}
\Opt = \inf_{x\in\R^n}\max\left(q(\gamma_-, x), q(\gamma_+, x)\right)
\quad\text{or}\quad
\Opt = \inf_{x\in\R^n}\set{q(\gamma_-, x):\, q_1(x)\leq 0}.
\end{align}
Note in the latter case that $A_1\succeq 0$ so that $q_1(x)\leq 0$ is a convex constraint.
Similarly, \eqref{eq:opt_equals_sup_inf} shows that the GTRS can be written as a concave maximization problem.

\begin{remark}\label{rem:gtrs_convex_reform}
The reformulation of the GTRS given in~\eqref{eq:gtrs_convex_reform} immediately suggests an algorithm for approximating $\Opt$:
Compute $\gamma_-$ (and if necessary $\gamma_+$) up to some accuracy and solve the resulting convex reformulation. Convergence guarantees along with rigorous error analyses for such an algorithm were previously explored by Wang and K{\i}l{\i}n\c{c}-Karzan~\cite{wang2020generalized}.
One drawback to this approach is that the convex functions $q(\gamma_-,x)$ and $q(\gamma_+,x)$ are, by construction, \emph{not} both strongly convex unless $A_0,A_1\succ 0$. Thus, in view of oracle lower bounds for first-order-methods~\cite[Chapter 2.1.2]{nesterov2018lectures}, one should not expect to achieve linear convergence rates via this approach.
Similarly, the reformulation of the GTRS given in \eqref{eq:opt_equals_sup_inf} immediately suggests an algorithm for approximating $\Opt$: apply a root-finding algorithm or binary search to find $\gamma^*$. This approach dates back to
Mor\'e and Sorenson~\cite{more1983computing} for the TRS and
Mor\'e~\cite{more1993generalizations} for the GTRS (see also~\cite{feng2012duality,adachi2019eigenvalue}). Unfortunately, theoretical convergence rates have not been established for algorithms of this form.\ije{\hfill\proofbox}{}
\end{remark}

We will combine both ideas above to construct strongly convex reformulations for instances of \eqref{eq:gtrs} possessing \emph{regularity}. Our notion of regularity will correspond to properties of $\mathbf{d}(\gamma)$ and its optimizers.
We will need the following notation.
\begin{definition}
For $\gamma\in\inter(\Gamma)$, define 
\begin{align*}
	x(\gamma)&\coloneqq -A(\gamma)^{-1}b(\gamma), \quad \nu(\gamma)\coloneqq q_1(x(\gamma)), \quad\text{and}\quad \mu(\gamma)\coloneqq \lambda_{\min}\left(A(\gamma)\right). \qedhere
\end{align*}
\end{definition}

The functions $\mathbf{d}(\gamma)$, $x(\gamma)$, and $\nu(\gamma)$ have been studied previously in the literature on algorithms for the TRS and the GTRS~\cite{adachi2019eigenvalue,more1983computing,more1993generalizations,feng2012duality}. In contrast to previous algorithms in this line of work, which propose methods for computing $\gamma^*$ to high accuracy, the algorithms we present in this paper will work with relatively inaccurate estimates of $\gamma^*$.	
Specifically, our algorithms are inspired by a key lemma, namely \cref{lem:gtrs_stability}, which says that if \eqref{eq:gtrs} has positive regularity, then the optimal solution to \eqref{eq:gtrs} is stable to inaccurate estimates of $\gamma^*$.
We begin by deriving some properties of $\mathbf{d}(\gamma)$ and its derivatives on $\inter(\Gamma)$.
\begin{lemma}
\label{lem:nu_is_derivative}
Suppose \cref{as:definiteness} holds.
If $\gamma\in\inter(\Gamma)$, then
\begin{align*}
\mathbf{d}(\gamma) = q(\gamma,x(\gamma))
\quad\text{and}\quad
\tfrac{d}{d\gamma}\mathbf{d}(\gamma) = \nu(\gamma).
\end{align*}
\end{lemma}
\begin{proof}
For $\gamma\in\inter(\Gamma)$, we have  $A(\gamma)\succ 0$ and thus $q(\gamma,x)$ is a strongly convex quadratic function in $x$. One may check that $\grad_x q(\gamma,x) = 2 \left(A(\gamma) x  + b(\gamma)\right)$, and thus $\mathbf{d}(\gamma) = q(\gamma,x(\gamma))$.

Next, from $\mathbf{d}(\gamma) = q(\gamma,x(\gamma))$ and $x(\gamma)= -A(\gamma)^{-1}b(\gamma)$, we deduce
\begin{align*}
\tfrac{d}{d\gamma}\mathbf{d}(\gamma) &= \tfrac{d}{d\gamma}\left(-b(\gamma)^\top A(\gamma)^{-1} b(\gamma) + c(\gamma)\right)\\
&= b(\gamma)^\top A(\gamma)^{-1}A_1 A(\gamma)^{-1}b(\gamma) -2b_1^\top A(\gamma)^{-1} b(\gamma) + c_1\\
&= q_1(x(\gamma)).\qedhere
\end{align*}
\end{proof}

\begin{lemma}
\label{lem:nu_derivatives}
	Suppose \cref{as:definiteness} holds. Let $\hat\gamma\in\inter(\Gamma)$, $P\coloneqq A(\hat\gamma)^{-1/2}$, and $\Delta\coloneqq (A_0P^2b_1 - A_1P^2b_0)$. Then, for $\gamma\in\inter(\Gamma)$,
	\begin{align*}
	\tfrac{d}{d\gamma}\nu(\gamma)&= -2\left(A_1x(\gamma) + b_1\right)^\top A(\gamma)^{-1} \left(A_1x(\gamma) + b_1\right)\\
	&= 
-2 \Delta^\top\left(A(\gamma)P^2A(\gamma)P^2A(\gamma)\right)^{-1}\Delta.
	\end{align*}
\end{lemma}
\begin{proof}
Starting from $\nu(\gamma)=q_1(x(\gamma))$, we compute
\begin{align*}
\tfrac{d}{d\gamma}\nu(\gamma) &= \ip{\grad_x q_1(x)\mid_{x = x(\gamma)}, \grad_\gamma x(\gamma)}\\
&= -2\ip{A_1x(\gamma) + b_1, A(\gamma)^{-1} (A_1x(\gamma) + b_1)}\\
&= -2 (A_1x(\gamma) + b_1)^\top A(\gamma)^{-1} (A_1x(\gamma) + b_1).
\end{align*}
Note also that
\begin{align*}
A_1 x(\gamma) + b_1 &= A(\gamma)A(\gamma)^{-1}b_1 - A_1A(\gamma)^{-1}b(\gamma)\\
&= \left(A_0A(\gamma)^{-1}b_1 +\gamma A_1A(\gamma)^{-1}b_1\right) - \left(A_1A(\gamma)^{-1}b_0 + \gamma A_1A(\gamma)^{-1}b_1\right)\\
&= A_0A(\gamma)^{-1}b_1 - A_1A(\gamma)^{-1}b_0.
\end{align*}

Next, suppose $\hat\gamma\in\inter(\Gamma)$
and let $P \coloneqq A(\hat\gamma)^{-1/2}$. Then,
$PA_0P$ and $PA_1P$ commute. Indeed, $PA_0P + \hat\gamma PA_1P = PA(\hat\gamma)P = I$.
Then,
\begin{align*}
A_0A(\gamma)^{-1}b_1 &= P^{-1} P A_0 P(PA(\gamma)P)^{-1} Pb_1\\
&= P^{-1}(PA(\gamma)P)^{-1} PA_0P^2 b_1\\
&= (A(\gamma)P^2)^{-1} A_0 P^2 b_1.
\end{align*}
Similarly, $A_1A(\gamma)^{-1}b_0 = (A(\gamma)P^2)^{-1} A_1 P^2 b_0$. We deduce
\begin{align*}
\tfrac{d}{d\gamma}\nu(\gamma) &= 
-2 \left(A_0 P^2 b_1 - A_1 P^2 b_0\right)^\top\left(A(\gamma)P^2A(\gamma)P^2A(\gamma)\right)^{-1}\left(A_0 P^2 b_1 - A_1 P^2 b_0\right).\qedhere
\end{align*}
\end{proof}

\begin{corollary}
\label{lem:gtrs_nu}
Suppose \cref{as:definiteness} holds. Then, $\nu(\gamma)$ is either a strictly decreasing or constant function of $\gamma$.
\end{corollary}
\begin{proof}
Fix $\hat\gamma\in\inter(\Gamma)$. By \cref{lem:nu_derivatives}, $\nu(\gamma)$ is strictly decreasing if $A_0A(\hat\gamma)^{-1}b_1 - A_1A(\hat\gamma)^{-1}b_0$ is nonzero. Else, $\nu(\gamma)$ is constant.
\end{proof}

\begin{corollary}
\label{cor:existence_optimizers}
Suppose \cref{as:definiteness} holds. Then, $\argmax_{\gamma\in\R_+}\mathbf{d}(\gamma)$ is either a unique point or is all of $\Gamma$. In the latter case, we furthermore have that $\Gamma$ is compact.
\end{corollary}
\begin{proof}
Note that by \cref{as:definiteness}, $\sup_{\gamma\in\R_+}\mathbf{d}(\gamma)$ is achieved. Indeed, as noted in \cref{obs:dual_properties}, $\mathbf{d}(\gamma)\to-\infty$ as $\gamma\to\infty$. Thus, $\argmax_{\gamma\in\R_+}\mathbf{d}(\gamma)$ is nonempty.

We will suppose that $\argmax_{\gamma\in\R_+}\mathbf{d}(\gamma)$ contains at least two points, $\gamma^{(1)}<\gamma^{(2)}$, and show that $\mathbf{d}(\gamma)$ is constant on $\Gamma$.
Note, by concavity of $\mathbf{d}(\gamma)$ and \cref{lem:nu_is_derivative}, we have that $\nu(\gamma) = 0$ for all $\gamma\in(\gamma^{(1)},\gamma^{(2)})$. By \cref{as:definiteness,lem:gtrs_nu}, $\nu(\gamma) = 0$ on all of $\inter(\Gamma)$ so that $\mathbf{d}(\gamma)$ is constant on $\inter(\Gamma)$. By continuity of $\mathbf{d}(\gamma)$ on $\Gamma$ (see \cref{obs:dual_properties}), $\mathbf{d}(\gamma)$ is then constant on all of $\Gamma$. This then implies that $\Gamma$ is compact as again by \cref{obs:dual_properties}, we have $\mathbf{d}(\gamma)\to-\infty$ as $\gamma\to\infty$.
\end{proof}

We now define our notion of regularity for the GTRS.
\begin{definition}
\label{def:regularity}
If $\sup_{\gamma\in\R_+}\mathbf{d}(\gamma)$ has a unique maximizer, then set $\gamma^*$ to be the unique maximizer. Otherwise, $\argmax_{\gamma\in\R_+}\mathbf{d}(\gamma) = \Gamma$ and let $\gamma^*\in \argmax_{\gamma\in\Gamma} \mu(\gamma)$.
Let $\mu^*\coloneqq \mu(\gamma^*)$.
We will say that the GTRS \eqref{eq:gtrs} has regularity $\mu^*$.
\end{definition}
\cref{cor:existence_optimizers} ensures that $\argmax_{\gamma\in\R_+}\mathbf{d}(\gamma)$ and $\mu^*$ in \cref{def:regularity} are well-defined. Note that, technically, $\gamma^*$ is \emph{not} well-defined if $\argmax_{\gamma\in\R_+}\mathbf{d}(\gamma) = \Gamma$ and $\mu(\gamma)$ has more than one maximizer. This is inconsequential and we may work with an arbitrary $\gamma\in\argmax_{\gamma\in\Gamma}\mu(\gamma)$. For concreteness, one may take $\gamma^*$ to be the minimum maximizer of $\mu(\gamma)$ in this case.

\begin{remark}\label{rem:trs_regularity_easy_hard}
We make a few observations on our definition of regularity and compare it to the so-called ``easy'' and ``hard'' cases of the trust-region subproblem (TRS).
Recall that the TRS is the special case of the GTRS~\eqref{eq:gtrs} where $q_1(x) = x^\top x - 1$, i.e., the constraint $q_1(x) \leq 0$ corresponds to the unit ball constraint $\norm{x}^2\leq 1$. We will assume that $A_0\not\succeq 0$.
Let $V\subseteq\R^n$ denote the eigenspace corresponding to $\lambda_{\min}(A_0)$.
The ``easy'' and ``hard'' cases of the TRS correspond to the cases $\Pi_V(b_0)\neq 0$ and $\Pi_V(b_0) = 0$ respectively. Here, $\Pi_V$ is the projection onto $V$.

In the ``easy'' case, it is possible to show that $\lim_{\gamma\searrow -\lambda_{\min}(A_0)} \mathbf{d}(\gamma) = -\infty$ so that $\gamma^*> -\lambda_{\min}(A_0)$ and $\mu^*>0$.
On the other hand, it is possible for $\mu^*>0$ even in the ``hard'' case. For example, taking $n =2$ and
\begin{align*}
A_0 = \begin{pmatrix}
	1 &\\
	& -1
\end{pmatrix},\qquad
b_0 = \begin{pmatrix}
	3\\
	0
\end{pmatrix},\qquad c_0 = 0,
\end{align*}
we have $\Gamma=[1,+\infty)$ and 
$\mathbf{d}(\gamma) = -9(1+\gamma)^{-1} -\gamma$ on $\inter(\Gamma)$. A simple computation then shows $\gamma^* = 2$ and $\mu^* = 1$.
We conclude that $\mu^* = 0$ implies the ``hard case'' but not necessarily vice versa.
\ije{\hfill\proofbox}{}
\end{remark}

We are now ready to present and prove our key lemma.

\begin{lemma}
	\label{lem:gtrs_stability}
	Suppose \cref{as:definiteness} holds, $\mu^*>0$ and the interval $[\gamma^{(1)},\gamma^{(2)}]\subseteq\R_+$ contains $\gamma^*$. Then, $\nu(\gamma^*)=0$ and $x(\gamma^*)$ is the unique optimizer of both \eqref{eq:gtrs} and
	\begin{align}
	\label{eq:gtrs_strong_convex_reformulation}
	\inf_{x\in\R^n}\max\left(q(\gamma^{(1)},x),q(\gamma^{(2)},x)\right).
	\end{align}
	In particular, taking $[\gamma^{(1)},\gamma^{(2)}]\subseteq\inter(\Gamma)$, we have that $x(\gamma^*)$ is the unique optimizer to the strongly convex problem \eqref{eq:gtrs_strong_convex_reformulation}.
\end{lemma}
\begin{proof}
We show that $x(\gamma^*)$ is the unique minimizer of \eqref{eq:gtrs_strong_convex_reformulation}.
Note that for all $x\in\R^n$, we have
\begin{align*}
\max\left(q(\gamma^{(1)},x),q(\gamma^{(2)},x)\right) \geq
q(\gamma^*, x) \geq \inf_{x\in\R^n} q(\gamma^*, x) = \mathbf{d}(\gamma^*),
\end{align*}
where the first inequality follows from the facts that $\gamma^*\in[\gamma^{(1)},\gamma^{(2)}]$ and $q(\gamma,x)$ is an affine function of $\gamma$.
On the other hand, as $\gamma^*\in\inter(\Gamma)$ is a maximizer of the smooth concave function $\mathbf{d}(\gamma)$ (see \cref{obs:dual_properties,def:regularity}), we have that $0=\tfrac{d}{d\gamma}\mathbf{d}(\gamma)|_{\gamma=\gamma^*}=\nu(\gamma^*) =q_1(x(\gamma^*))$ where the second equation follows from \cref{lem:nu_is_derivative}. 
Then, $q_1(x(\gamma^*))=0$ implies that $q(\gamma,x(\gamma^*))=q_0(x(\gamma^*))$ for any $\gamma$. 
Hence, we deduce that
\begin{align*}
\max\left(q(\gamma^{(1)},x(\gamma^*)),q(\gamma^{(2)},x(\gamma^*))\right) =
q(\gamma^*, x(\gamma^*)) = \mathbf{d}(\gamma^*)
\end{align*}
so that $x(\gamma^*)$ is a minimizer of \eqref{eq:gtrs_strong_convex_reformulation}.
Uniqueness of $x(\gamma^*)$ then follows from the fact that $q(\gamma^*, x)$ is a strongly convex function of $x$ and it lower bounds the objective function $\max\left(q(\gamma^{(1)},x),q(\gamma^{(2)},x)\right)$ of \eqref{eq:gtrs_strong_convex_reformulation}.

The proof that $x(\gamma^*)$ is the unique optimizer of \eqref{eq:gtrs} follows verbatim using the lower bound: $q_0(x)\geq q(\gamma^*,x)$ for all $x\in\R^n$ such that $q_1(x) \leq 0$.
\end{proof}

\section{Algorithms for the GTRS}
\label{sec:constructing_sc_reform}
We now turn to the GTRS and present an approach for computing $\Opt$ that exploits regularity in \eqref{eq:gtrs}.
Our approach will consist of two parts: constructing a convex reformulation of \eqref{eq:gtrs} and solving the convex reformulation.
In conjunction, these two pieces will allow us to achieve \emph{linear} convergence rates for \eqref{eq:gtrs} whenever $\mu^*>0$.

Similar to other recent papers on the GTRS~\cite{wang2020generalized,jiang2020linear}, we will assume that we are given as input the problem data $(A_0, A_1, b_0, b_1, c_0, c_1)$, regularity parameters $(\xi,\zeta,\hat\gamma)$, and error and failure parameters $(\epsilon,p)$. We will make the following assumption on our input data.
\begin{assumption}
\label{as:alg_gtrs}
Suppose that for both $i\in\set{0,1}$, $A_i$ has at least one negative eigenvalue, $\norm{A_i},\norm{b_i},\abs{c_i}\leq 1$. Let $N$ denote the number of nonzero entries in $A_0$ and $A_1$ combined and assume $N\geq n$.
Furthermore, suppose $\gamma_+ \leq \zeta$, $A(\hat\gamma)\succeq \xi I$, $0<\xi\leq 1\leq \zeta$, and $\epsilon,p\in(0,1)$.
\end{assumption}
These assumptions are relatively minor. Indeed, $N\geq n$ without loss of generality. Furthermore, if any of the norms $\norm{A_i}, \norm{b_i},\abs{c_i}$ are larger than $1$, we may scale the entire function $q_i(x)$ until \cref{as:alg_gtrs} holds.

\begin{remark}\label{rem:regularity}
	The regularity parameters $\xi$ and $\zeta$ will appear in our error and running time bounds.
	We make no attempt to optimize constants in these bounds and will routinely apply the following bounds (following from \cref{as:alg_gtrs}) for $\gamma\in\Gamma$: $\norm{A(\gamma)},\norm{b(\gamma)},\abs{c(\gamma)} \leq 1 + \zeta \leq 2\zeta$.\ije{\hfill\proofbox}{}
\end{remark}

Our first algorithm, \texttt{ConstructReform} (\cref{alg:construct_reform_gtrs}), will attempt to construct a convex reformulation of \eqref{eq:gtrs} with strong convexity on the order of $\min(\mu^*,\xi)$. 
Note, however, that it may be undesirable to compute this reformulation if $\min(\mu^*,\xi) \lesssim \epsilon$. In view of this, we define
\begin{align*}
\phi\coloneqq \max\left((\min(\mu^*,\xi),\, \epsilon\xi^4/\zeta^4\right).
\end{align*}
To understand this quantity, note that $[\epsilon\xi^4/\zeta^4, \xi]$ is an interval and that $\phi$ is the closest point to $\mu^*$ in this interval.
Then, \texttt{ConstructReform}, will either output an exact \emph{strongly convex} reformulation of \eqref{eq:gtrs} with strong convexity on the order $\phi$ or an $\epsilon$-approximate optimizer.
In the former case, we will then apply our second algorithm, \texttt{SolveRegular} (\cref{alg:solveRegular}), to compute an $\epsilon$-approximate optimizer.

\begin{remark}\label{rem:construct_reform_only_once}
\texttt{ConstructReform} needs to successfully output an exact strongly convex reformulation only \emph{once}. Specifically, if after computing a strongly convex reformulation of \eqref{eq:gtrs}, the value of $\epsilon>0$ is changed, we may skip running \texttt{ConstructReform} a second time and simply run \texttt{SolveRegular} with the new value of $\epsilon>0$.
\ije{\hfill\proofbox}{}
\end{remark}

\cref{app:useful_procedures} contains useful algorithms and guarantees from the literature that we will use as building blocks in \texttt{ConstructReform} and \texttt{SolveRegular}. Specifically, \cref{app:useful_procedures} recalls the running time of the conjugate gradient algorithm for minimizing a quadratic function (\cref{lem:conjugate_grad}), the running time of the Lanczos method for finding a minimum eigenvalue (\cref{lemma:ApproxEig}), and the running time of Nesterov's accelerated gradient descent method for minimax problems applied to the maximum of two quadratic functions (\cref{lem:acc_minimax}). We additionally present \texttt{ApproxGammaLeft}, a minor modification of \cite[Algorithm 2]{wang2020generalized} for finding an aggregation weight $\gamma\leq \hat\gamma$ such that $\mu(\gamma)$ falls in a specified range, and \texttt{ApproxNu}, a restatement of the conjugate gradient guarantee for the purpose of approximating $\nu(\gamma)$. We state the guarantees of \texttt{ApproxGammaLeft} and \texttt{ApproxNu} below and leave their proofs to \cref{app:useful_procedures}.

\begin{restatable}{lemma}{lemapproxgamma}
\label{lem:approx_gamma}
Suppose \cref{as:alg_gtrs} holds, $\mu\in(0,\xi)$ and $p\in(0,1)$. Then, with probability at least $1-p$, \texttt{ApproxGammaLeft}$(\mu,p)$ (\cref{alg:ApproxGamma}) returns $(\gamma,v)$ such that $\gamma\leq \hat\gamma$ and $v$ is a unit vector satisfying $\mu/2 \leq \mu(\gamma)\leq v^\top A(\gamma) v \leq \mu$ in time
\begin{align*}
\tilde O\left(\tfrac{N\sqrt{\zeta}}{\sqrt{\mu}}\log\left(\tfrac{n}{p}\right)\log\left(\tfrac{\zeta}{\mu}\right)\right).
\end{align*}
\end{restatable}

\begin{restatable}{lemma}{lemapproxnu}
\label{lem:approxnu}
Suppose \cref{as:alg_gtrs} holds, $\mu\in(0,\xi]$, $\delta\in(0,1)$, and $A(\gamma)\succeq \mu I$. Then \texttt{ApproxNu}$(\mu,\delta,\gamma)$ (\cref{alg:approxnu}) returns $(\tilde x,\tilde\nu)$ such that
$\norm{\tilde x - x(\gamma)} \leq \mu\delta/10\zeta$, and $
\tilde\nu = q_1(\tilde x) \in [\nu(\gamma) \pm \delta]$ in time
\begin{align*}
O\left(\tfrac{N\sqrt{\zeta}}{\sqrt{\mu}}\log\left(\tfrac{\zeta}{\mu\delta}\right)\right).
\end{align*}
\end{restatable}

\subsection{Constructing a strongly convex reformulation}

We present and analyze \texttt{ConstructReform} (\cref{alg:construct_reform_gtrs}).
For the sake of presentation, we break \texttt{ConstructReform} into the following parts.

\begin{algorithm}
\caption{\texttt{ConstructReform}}
\label{alg:construct_reform_gtrs}
	Given $(A_0,A_1,b_0,b_1,c_0,c_1)$, $(\xi,\zeta,\hat\gamma)$ and $\epsilon,p\in(0,1)$ satisfying \cref{as:alg_gtrs}
	{\small\begin{enumerate}[itemsep=0pt,topsep=0pt,parsep=0pt]
		\item Set $\gamma_0 =\hat\gamma$, $\mu_0 = \xi$
		\item Set $(x_0, \nu_0) = \texttt{ApproxNu}(\mu_0, \epsilon/(4\zeta), \gamma_0)$
		\item If $\nu_0 + \epsilon/(4\zeta) < 0$, run \texttt{CRLeft} (\cref{alg:construct_reform_left})
		\item Else if $\nu_0 - \epsilon/(4\zeta) > 0$, run \texttt{CRRight}
		\item Else, run \texttt{CRMid} (\cref{alg:constructReformMid})
	\end{enumerate}}
\end{algorithm}

We will say that \texttt{ConstructReform} (similarly, \texttt{CRLeft}, \texttt{CRMid}, and \texttt{CRRight}) \emph{succeeds} if it either outputs:
\begin{itemize}[itemsep=0pt,topsep=0pt,parsep=0pt]
	\item \texttt{"regular"}, $\gamma^{(1)}$, $\gamma^{(2)}$, $\tilde \mu$ such that $\gamma^*\in[\gamma^{(1)},\gamma^{(2)}]$ and $\mu(\gamma^{(i)})\geq \tilde\mu \geq \min(\mu^*,\xi)/8$,
	\item \texttt{"maybe regular"}, $x$ such that $x$ is an $\epsilon$-approximate optimizer, or
	\item \texttt{"not regular"}, $x$ such that $x$ is an $\epsilon$-approximate optimizer.
\end{itemize}

The remainder of this subsection proves the following guarantee.

\begin{proposition}
\label{thm:alg_gtrs_correctness}
Suppose \cref{as:alg_gtrs} holds. 
With probability at least $1-p$, \texttt{ConstructReform} (\cref{alg:construct_reform_gtrs}) succeeds and runs in time
\begin{align*}
\tilde O\left(\frac{N\sqrt{\zeta}}{\sqrt{\phi}}\log\left(\frac{1}{\phi}\right) \log \left(\frac{n}{p}\right)\log\left(\frac{\zeta}{\epsilon\xi}\right)\right).
\end{align*}
\end{proposition}

\cref{thm:alg_gtrs_correctness} will follow as an immediate corollary to the the corresponding guarantees for \texttt{CRLeft}, \texttt{CRRight}, and \texttt{CRMid}.
The steps and analysis of \texttt{CRRight} are analogous to that of \texttt{CRLeft} and are omitted.

Our algorithms will attempt to binary search for $\gamma^*$ using the sign of $\nu(\gamma)$. Unfortunately, as we can only approximate $\nu(\gamma)$ up to some accuracy, we will need to argue how to handle situations where our approximation of $\nu(\gamma)$ is close to zero.
\begin{lemma}
\label{lem:approx_nu_almost_zero}
Suppose \cref{as:alg_gtrs} holds, $\mu\in(0,\xi]$, $\epsilon\in(0,1)$, and $A(\gamma)\succeq \mu I$. 
Let $(\tilde x, \tilde \nu) = \texttt{ApproxNu}(\mu,\epsilon/(4\zeta),\gamma)$. If $\tilde\nu \in[\pm \epsilon/(4\zeta)]$, then $\tilde x$ is an $\epsilon$-approximate optimizer of \eqref{eq:gtrs}.
\end{lemma}
\begin{proof}
By \cref{lem:approxnu}, we have that $q_1(x(\gamma)) = \nu(\gamma) \in [\tilde \nu \pm \epsilon/(4\zeta)]\subseteq [\pm \epsilon/(2\zeta)]$ where the last containment follows from $\tilde\nu \in[\pm \epsilon/(4\zeta)]$ in the premise of the lemma. Also, note that
\begin{align*}
q_0(x(\gamma))= q(\gamma, x(\gamma)) - \gamma \nu(\gamma) \leq \Opt + \epsilon/2.
\end{align*}
Here, the inequality follows from the bounds $\nu(\gamma) \in[\pm\epsilon/(2\zeta)]$, $\gamma\leq \gamma_+\leq \zeta$ (as $A(\gamma)\succeq 0$ we have $\gamma\in\Gamma$ and \cref{as:alg_gtrs} ensures $\gamma_+\leq \zeta$), and $q(\gamma,x(\gamma)) = \mathbf{d}(\gamma)\leq \Opt$.
Thus, we deduce that $x(\gamma)$ is an $\epsilon/2$-approximate optimizer.

Next, by \cref{lem:approxnu}, we have $\norm{x(\gamma) - \tilde x} \leq \epsilon \mu/(40\zeta^2)$. 
Note that $\norm{x(\gamma)}=\norm{-A(\gamma)^{-1}b(\gamma)}\leq \norm{A(\gamma)^{-1}}\norm{b(\gamma)}\leq 2\zeta/\mu$ where the last inequality follows from $A(\gamma)\succeq \mu I$ and $\norm{b(\gamma)}\leq 2\zeta$ (implied by \cref{rem:regularity}).
Considering \cref{as:alg_gtrs} and applying \cref{lem:quadratic_error} with the bounds $\norm{x(\gamma)}\leq 2\zeta/\mu$ and $\norm{x(\gamma) - \tilde x} \leq \epsilon \mu/(40\zeta^2)$, we arrive at 
\begin{gather*}
	q_0(\tilde x) \leq q_0(x(\gamma)) + 5\epsilon \frac{\mu}{40\zeta^2}\frac{2\zeta}{\mu} \leq \Opt + \frac{\epsilon}{2} + \frac{\epsilon}{4\zeta}\leq \Opt + \epsilon\\
	q_1(\tilde x) \leq q_1(x(\gamma)) + 5\epsilon \frac{\mu}{40\zeta^2}\frac{2\zeta}{\mu} \leq \frac{\epsilon}{2} + \frac{\epsilon}{4\zeta}\leq \epsilon.\qedhere
\end{gather*}
\end{proof}

\begin{remark}
In contrast to the TRS setting,
where it is possible to show that $\nu(\gamma)$ ``grows quickly'' around $\gamma^*$,
in the GTRS setting, $\nu(\gamma)$ may be ``arbitrarily flat''.
In particular, it may not be possible to determine the sign of $\nu(\gamma)$ given only an inaccurate estimate.
Correspondingly, \texttt{ConstructReform} may
fail to differentiate between \texttt{"regular"} and \texttt{"not regular"} instances and return \texttt{"maybe regular"}.
In view of \cref{rem:construct_reform_only_once}, we will think of \texttt{"maybe regular"} outputs as being less desirable than \texttt{"regular"} outputs.
We will explore this issue in further detail in \cref{sec:maybe_regular} and show that \texttt{ConstructReform} does not output \texttt{"maybe regular"} as long as the GTRS instance satisfies a \emph{coherence} condition.\ije{\hfill\proofbox}{}
\end{remark}

\subsubsection{Analysis of \texttt{CRLeft}}
\label{subsec:construct_reform_left}

\cref{alg:construct_reform_gtrs} calls \texttt{CRLeft} if $\nu_0 + \epsilon/4\zeta < 0$. Note that in this case, from \cref{lem:approxnu} we have $\nu(\hat\gamma) = \nu(\gamma_0) \in [ \nu_0 \pm \epsilon/(4\zeta)]$ which implies $\nu(\hat\gamma)  < 0$.

\begin{algorithm}
\caption{\texttt{CRLeft}}
	\label{alg:construct_reform_left}
	{\small\begin{enumerate}[itemsep=0pt,topsep=0pt,parsep=0pt]
			\item Let $T \coloneqq \ceil{\log\left(\tfrac{3200\zeta^4}{\epsilon\xi^3}\right)}$. For $t = 1,\dots,T$,
			\begin{enumerate}[itemsep=0pt,topsep=0pt,parsep=0pt]
				\item Set $\mu_t = 2^{-t}\xi$
				\item Set $(\gamma_t, v_t) = \texttt{ApproxGammaLeft}(\mu_t, p/T)$
				\item Set $(x_t,\nu_t) = \texttt{ApproxNu}(\mu_t/2, \epsilon/(4\zeta),\gamma_t)$
				\item If $\nu_t - \epsilon/(4\zeta) > 0$, return \texttt{"regular"}, $\gamma_{t}$, $\hat\gamma$, $\mu_t/4$
				\item Else if $\nu_t\in[-\epsilon/(4\zeta),\, \epsilon/(4\zeta)]$
				\begin{enumerate}[itemsep=0pt,topsep=0pt,parsep=0pt]
					\item Set $\gamma' \coloneqq \gamma_t - \mu_t/4$
					\item Set $(x', \nu') = \texttt{ApproxNu}(\mu_t/4,\epsilon/(4\zeta), \gamma')$
					\item If $\nu' - \epsilon /(4\zeta) > 0$, return \texttt{"regular"}, $\gamma'$, $\hat\gamma$, $\mu_t/4$
					\item Else, return \texttt{"maybe regular"}, $x_t$
				\end{enumerate}
			\end{enumerate}
			\item If necessary, negate $v_T$ so that $\ip{v_T, A(\gamma_T) x_T + b(\gamma_T)} \leq 0$. Let $\alpha>0$ such that $q_1(x_T+\alpha v_T) = 0$, return \texttt{"not regular"}, $x_T+\alpha v_T$.
		\end{enumerate}}
\end{algorithm}

\begin{proposition}
\label{prop:nu_hat_gamma_positive_correctness}
Suppose \cref{as:alg_gtrs} holds.
	With probability at least $1 - p$, \texttt{CRLeft} (\cref{alg:construct_reform_left}) succeeds and runs in time
	\begin{align*}
	\tilde O\left(\tfrac{N\sqrt{\zeta}}{\sqrt{\phi}}\log\left(\tfrac{1}{\phi}\right)\log\left(\tfrac{n}{p}\right)\log\left(\tfrac{\zeta}{\epsilon\xi}\right)\right).
	\end{align*}
\end{proposition}
\begin{proof}
We condition on step 1.(b) of \texttt{CRLeft} succeeding in every iteration. This happens with probability at least $1-p$.

We begin with the running time. Note that by \cref{lem:approxnu,lem:approx_gamma} and $\mu_t = 2^{-t}\xi$ (from step 1.(a)), iteration $t$ of line 1 runs in time
\begin{align*}
\tilde O\left(\tfrac{N\sqrt{\zeta}}{\sqrt{\mu_t}}\log\left(\tfrac{n}{p}\right)\log\left(\tfrac{\zeta}{\epsilon\xi}\right)\right).
\end{align*}
It suffices then to show that $\mu_t = \Omega(\phi)$ in every iteration before \texttt{CRLeft} outputs.
Noting that $\mu_t \geq \mu_T = \Omega(\epsilon\xi^4/\zeta^4)$, we may instead show that $\mu_t = \Omega(\max(\mu^*,\xi))$ in the iteration at which \texttt{CRLeft} outputs.

It remains to show that the output of \texttt{CRLeft} satisfies the success criteria and that $\mu_t = \Omega(\max(\mu^*,\xi))$ for the iteration $t$ at which \texttt{CRLeft} outputs.
We split the remainder of the proof into three parts depending on which line \texttt{CRLeft} returns on.

\paragraph{Case 1: \texttt{CRLeft} terminates on either line 1.(d) or 1.(e).iii in iteration $t$}

Let $\tilde\gamma \coloneqq \gamma_t$ in the first case and $\tilde \gamma\coloneqq \gamma'$ in the second.
As \texttt{CRLeft} did not terminate at time $t -1$, we have that $\nu(\gamma_{t-1}) < 0$.
Indeed, if $\nu(\gamma_{t-1})\geq 0$, then $\nu_{t-1}\geq - \epsilon/4\zeta$ by \cref{lem:approxnu}.
Then, $\nu(\tilde\gamma) > 0 > \nu(\gamma_{t-1})$.
We deduce by the fact that $\mathbf{d}(\gamma)$ is concave and \cref{lem:nu_is_derivative} that $\gamma^*\in[\tilde\gamma,\gamma_{t-1}]\subseteq [\tilde\gamma,\hat\gamma]$.
By construction in line 1.(b), we have that $\mu(\tilde\gamma) \geq \mu_t /4$.

It remains to show that $\mu_t \geq \min(\mu^*,\xi)/2$.
This holds if $t = 1$, as then $\mu_1 = \xi/2$ by line 1.(a).
On the other hand, if $t > 1$, then $\mu(\gamma)$ is an increasing function on the interval $(\infty, \gamma_{t-1}]$. Indeed, this follows as $\gamma_{t-1} \leq \hat\gamma$, $\mu(\gamma_{t-1}) \leq \xi/2 < \mu(\hat\gamma)$, and $\mu(\gamma)=\lambda_{\min}(A_0+\gamma A_1)$ is a concave function of $\gamma$. Then, from $\gamma^*\in[\tilde\gamma,\gamma_{t-1}]$, we deduce that
\begin{align*}
\mu^* = \mu(\gamma^*) \leq  \mu(\gamma_{t-1}) \leq \mu_{t-1} = 2\mu_t,
\end{align*}
where the last inequality follows from line 1.(b).

\paragraph{Case 2: \texttt{CRLeft} terminates on line 1.(e).iv in iteration $t$}

In this case, we have that $(x_t,\nu_t) = \texttt{ApproxNu}(\mu_t/2, \epsilon/(4\zeta), \gamma_t)$ satisfies $\nu_t \in[\pm \epsilon/(4\zeta)]$. By \cref{lem:approx_nu_almost_zero}, we have that $x_t$ is an $\epsilon$-approximate optimizer.
It remains to note that the second paragraph of Case 1 holds in this case verbatim so that $\mu_t \geq \min(\mu^*,\xi)/2$.

\paragraph{Case 3: \texttt{CRLeft} terminates on line 2}
Note that $q_1(x_T) = \nu_T < 0$ holds by line 1.(c),  \cref{lem:approxnu}, and the fact that \texttt{CRLeft} did not terminate in a prior line. Furthermore,
\begin{align*}
v_T^\top A_1 v_T = v_T^\top\left(\frac{A(\hat\gamma) - A(\gamma_T)}{\hat\gamma - \gamma_T}\right)v_T \geq \frac{\xi - \mu_T}{\zeta}\geq \frac{\xi}{2\zeta}>0,
\end{align*}
where the first inequality follows from $\zeta\geq \hat \gamma$ (by \cref{as:alg_gtrs}), $v_T^\top A(\gamma_T) v_T \leq \mu_T$ (by line 1.(b) and \cref{lem:approx_gamma}) and $v_T^\top A(\hat \gamma) v_T = v_T^\top A_0 v_T + \hat \gamma \geq \xi$ (by \cref{as:alg_gtrs} and $\hat\gamma\geq0$), and the second inequality follows from $\mu_T=2^{-T}\xi$ by line 1(a).
This then implies that $\alpha$ in line 2 is well-defined. Thus, by construction in line 2, $q_1(x_T+\alpha v_T) = 0$.
Our goal is to show that
\begin{align*}
q_0(x_T+\alpha v_T) &= q(\gamma_T, x_T + \alpha v_T) \leq q(\gamma_T, x_T) + \alpha^2 \mu_T \leq \Opt + \epsilon.
\end{align*}

The following sequence of inequalities allows us to bound $\norm{x(\gamma_T)}$:
\begin{align*}
\xi \norm{x(\gamma_T)}^2 - 4\zeta \norm{x(\gamma_T)} - 2\zeta \leq q(\hat\gamma, x(\gamma_T)) \leq q(\gamma_T,x(\gamma_T)) \leq \Opt.
\end{align*}
Here, the first inequality follows from $A(\hat\gamma)\succeq \xi I$, $\norm{b(\hat\gamma)}\leq 2\zeta$ and $\abs{c(\hat\gamma)}\leq 2\zeta$, the second inequality follows as $0\geq \nu(\gamma_T) =q_1(x(\gamma_T))$ (by line 1.(c), \cref{lem:approxnu} and the fact that \texttt{CRLeft} terminates on line 2) and $\hat\gamma\geq \gamma_T$, the third inequality follows as $q(\gamma_T, x(\gamma_T)) = \mathbf{d}(\gamma_T)\leq\Opt$ (by \cref{lem:nu_is_derivative}). Then, taking $x = 0$ in the expression $\Opt = \inf_x\sup_{\gamma\in\Gamma}q(\gamma,x)$ gives $\Opt \leq 2\zeta$. Applying \cref{lem:quadraticRootsBound} to $\xi \norm{x(\gamma_T)}^2 - 4\zeta \norm{x(\gamma_T)} - 4\zeta \leq 0$ gives $\norm{x(\gamma_T)} \leq (2\sqrt{2}+2)\zeta/\xi\leq5\zeta/\xi$, and by \cref{as:alg_gtrs} and line 1.(c) we have $\norm{A_1x_T + b_1}\leq \norm{A_1}\left(\norm{x(\gamma_T)} + \norm{x_T - x(\gamma_T)}\right) + \norm{b_1} \leq (5\zeta/\xi + 1)+1\leq 7\zeta/\xi$.

Next, we may bound
\begin{align*}
q(\gamma_T, x_T) &\leq q(\gamma_T, x(\gamma_T)) + \norm{A(\gamma_T)}\norm{x(\gamma_T) - x_T}^2\\
&\leq \Opt + (2\zeta)\left(\frac{\mu_T\epsilon}{80\zeta^2}\right)^2 \leq \Opt + \epsilon/2.
\end{align*}

Similarly, $\nu(\gamma_T) \geq \nu(\hat\gamma) = q_1(x(\hat\gamma)) \geq -\norm{x(\hat\gamma)}^2 - 2 \norm{x(\hat\gamma)} - 1\geq  - (3\zeta/\xi)^2$, where the first inequality follows from \cref{lem:gtrs_nu} and the last from the bound $\norm{x(\hat\gamma)} \leq 2\zeta/\xi$. We deduce that $0 \geq q_1(x_T) \geq \nu(\gamma_T) - \epsilon/(4\zeta) \geq -10\zeta^2/\xi^2$.
By line 2 and applying \cref{lem:quadraticRootsBound}, we have that $\alpha \leq 40 \zeta^2/\xi^2$.

We conclude that $\alpha^2 \mu_T \leq \alpha^2 \frac{\epsilon\xi^4}{3200 \zeta^4} \leq \frac{\epsilon}{2}$ so that $q_0(x_T+\alpha v_T) = q(\gamma_T, x_T+\alpha v_T) \leq q(\gamma_T,x_T) + \alpha^2\mu_T\leq \Opt + \epsilon$, where the equation follows from the definition of $\alpha$ in line 2.

It remains to note that as $\nu(\gamma_T) < 0$, \cref{lem:gtrs_nu} implies $\gamma^* \leq \gamma_T$ and $\mu^* = \mu(\gamma^*) \leq \mu(\gamma_T) \leq \mu_T$.
\end{proof}

\subsubsection{Analysis of \texttt{CRMid}}
\label{subsec:construct_reform_mid}

\cref{alg:construct_reform_gtrs} calls \texttt{CRMid} if $\nu_0 \in[-\epsilon/(4\zeta),\, \epsilon/(4\zeta)]$. Note that in this case, we may deduce $\abs{\nu(\hat\gamma)} = \abs{\nu(\gamma_0)} \leq \epsilon/(2\zeta)$.

\begin{algorithm}
	\caption{\texttt{CRMid}}
	\label{alg:constructReformMid}
	\leavevmode
	{\small\begin{enumerate}[itemsep=0pt,topsep=0pt,parsep=0pt]
		\item Let $\gamma' \coloneqq \gamma_0 - \xi/2$ and $\gamma'' \coloneqq \gamma_0 + \xi/2$
		\item Set $(x',\nu') = \texttt{ApproxNu}(\gamma', \epsilon/(4\zeta))$
		\item Set $(x'',\nu'') = \texttt{ApproxNu}(\gamma'', \epsilon/(4\zeta))$
		\item If $\nu'-\epsilon/(4\zeta) >0 > \nu''+\epsilon/(4\zeta)$, return \texttt{"regular"}, $\gamma'$, $\gamma''$, $\xi/2$
		\item Else if $\nu'-\epsilon/(4\zeta)\leq 0$, return \texttt{"maybe regular"}, $x_0$
		\item Else, return \texttt{"maybe regular"}, $x_0$
	\end{enumerate}}
\end{algorithm}

\begin{proposition}
\label{prop:nu_hat_gamma_zero_correctness}
	Suppose \cref{as:alg_gtrs} holds. Then, \texttt{CRMid} (\cref{alg:constructReformMid}) succeeds and runs in time
	\begin{align*}
	O\left(\tfrac{N\sqrt{\zeta}}{\sqrt{\xi}}\log\left(\tfrac{\zeta}{\epsilon\xi}\right)\right).
	\end{align*}
\end{proposition}
\begin{proof}
Suppose \texttt{CRMid} returns on line 4. Then, by \cref{lem:approxnu} and lines 2 and 3 we have $\nu(\gamma') > 0 > \nu(\gamma'')$. We deduce by the fact that $\mathbf{d}(\gamma)$ is concave and \cref{lem:nu_is_derivative}, that $\gamma^*\in[\gamma',\gamma'']$.
Furthermore, $\mu(\hat\gamma \pm \xi/2) \geq \mu(\hat\gamma) - \xi/2\geq \xi/2$ as $\mu$ is $1$-Lipschitz and $\mu(\hat\gamma)\geq \xi$.

If, \texttt{CRMid} returns on lines 5 or 6, then $(x_0, \nu_0) = \texttt{ApproxNu}(\mu_0, \epsilon/(4\zeta), \gamma_0)$ satisfies $\nu_0\in[\pm \epsilon/(4\zeta)]$. By \cref{lem:approx_nu_almost_zero}, we have that $x_0$ is an $\epsilon$-approximate optimizer.

The running time of \texttt{CRMid} follows from \cref{lem:conjugate_grad}.
\end{proof}

\subsection{Solving the convex reformulation}
\label{sec:solving_reformulations}

\begin{algorithm}
	\caption{\texttt{SolveRegular}}
	\label{alg:solveRegular}
	Given $\gamma^{(1)},\gamma^{(2)},\tilde \mu$ such that $\gamma^*\in[\gamma^{(1)},\gamma^{(2)}]$ and $\min_{i\in[2]}\set{\mu(\gamma^{(i)})}\geq \tilde\mu >0$
	{\small \begin{enumerate}[parsep=0pt,topsep=0pt,itemsep=0pt]
		\item Apply Nesterov's accelerated minimax scheme for strongly convex smooth quadratic functions to compute a $\tilde\mu \left(\epsilon\tilde\mu/10\zeta\right)^2$-optimal solution $\bar x$ to
		\begin{align*}
		\min_{x\in\R^n}\max\left(q(\gamma^{(1)},x), q(\gamma^{(2)},x)\right)
		\end{align*}
		\item Return $\bar x$
	\end{enumerate}}
\end{algorithm}

\begin{proposition}
\label{prop:Case1_time}
Suppose \cref{as:alg_gtrs} holds and $\tilde \mu \in(0,\xi]$.
Then, \texttt{Solve\-Regular} (\cref{alg:solveRegular}) computes an $\epsilon$-approximate solution to \eqref{eq:gtrs} in time
\begin{align*}
O\left(\tfrac{N\sqrt{\zeta}}{\sqrt{\tilde\mu}} \log\left(\tfrac{\zeta}{\epsilon\tilde\mu}\right)\right).
\end{align*}
\end{proposition}
\begin{proof}
	For notational simplicity, let $q_{\max}(x) \coloneqq \max\left(q(\gamma^{(1)}, x),\, q(\gamma^{(2)}, x)\right)$.
	Let $x^* \coloneqq x(\gamma^*)$. Recall that $q_0(x^*) = \Opt$, $q_1(x^*) = 0$, and $q_{\max}(x^*) = \Opt$.
	Then, by definition of $\mu^*$ in \cref{def:regularity} and strong convexity of $q(\gamma^*,x)$, we have
	\begin{align*}
	\tilde \mu \norm{x^* - \bar x}^2&\leq \mu^* \norm{x^* - \bar x}^2\leq q(\gamma^*, \bar x) - q(\gamma^*, x^*) = q(\gamma^*, \bar x) -\Opt\\
	&\leq q_{\max}(\bar x) - \Opt \leq \tilde \mu \left(\frac{\epsilon\tilde\mu}{10\zeta}\right)^2.
	\end{align*}
	Rearranging, we may bound $\norm{x^* - \bar x} \leq \tfrac{\epsilon\tilde\mu}{10\zeta}$.
	Furthermore, $\norm{x^*} = \norm{x(\gamma^*)}= \norm{-A(\gamma^*)^{-1}b(\gamma^*)}$ so that $\norm{x^*} \leq 2\zeta/{\tilde\mu}$ holds by \cref{as:alg_gtrs}.

	Then, as $\epsilon\tilde\mu/(10\zeta) \leq 1$ and  $2\zeta/\tilde \mu \geq 1$ (by definition of $\tilde \mu$ and \cref{as:alg_gtrs}), we can apply \cref{lem:quadratic_error} to get
	\begin{gather*}
		q_0(\bar x) \leq q_0(x^*) + 5\epsilon\frac{\tilde\mu}{10\zeta}\frac{2\zeta}{\tilde\mu}  = \Opt + \epsilon\\
		q_1(\bar x) \leq q_1(x^*) + 5\epsilon\frac{\tilde\mu}{10\zeta}\frac{2\zeta}{\tilde\mu}  = \epsilon.
	\end{gather*}

	The running time follows from \cref{lem:acc_minimax}.
\end{proof}

\subsection{Putting the pieces together}
\label{subsec:pieces_together}
The following theorem states the guarantee for applying \texttt{ConstructReform} (\cref{alg:construct_reform_gtrs}) and \texttt{SolveRegular} (\cref{alg:solveRegular}). This guarantee follows as a corollary to \cref{prop:nu_hat_gamma_positive_correctness,prop:nu_hat_gamma_zero_correctness,prop:Case1_time}
\begin{theorem}
\label{thm:overall_time}
Suppose \cref{as:alg_gtrs} holds. Then with probability $1-p$, the procedure outlined above returns an $\epsilon$-approximate solution to \eqref{eq:gtrs} in time
\begin{align*}
\tilde O\left(\tfrac{N}{\sqrt{\phi}}\log\left(\tfrac{1}{\phi}\right) \log\left(\tfrac{n}{p}\right)\log\left(\tfrac{\zeta}{\epsilon\xi}\right)\right).
\end{align*}
\end{theorem} 

\subsection{Revisiting \texttt{"maybe regular"} outputs}
\label{sec:maybe_regular}
We revisit \texttt{Construct\-Reform} (\cref{alg:construct_reform_gtrs}) and show that \texttt{Construct\-Reform} does not output \texttt{"maybe regular"} on a successful run as long as a coherence condition is satisfied.

The following examples shows that in the GTRS setting, $\nu(\gamma)$ may grow arbitrarily slowly near $\gamma^*$.
\begin{example}
Let $n = 2$ and $\epsilon\in(0,1/4)$ and set
\begin{align*}
A_0 = \begin{pmatrix}
	1\\
	& - 1/2
\end{pmatrix},\quad
A_1 = \begin{pmatrix}
	-1\\& 1
\end{pmatrix},\quad
b_0 = \epsilon \cdot e_1,\quad
b_1	=	0,\quad
c_0 = 0,\quad
c_1 = 16\epsilon^2.
\end{align*}
Note that $\Gamma = [1/2, 1]$ and $A(3/4) = I/4$ so that \cref{as:alg_gtrs} holds with $\xi = 1/4$ and $\zeta = 1$. Then, we have
\begin{align*}
x(\gamma) = -\frac{\epsilon}{1-\gamma}e_1,\quad
\nu(\gamma) = \epsilon^2\left(16 - \frac{1}{(1-\gamma)^2}\right),\,\forall \gamma\in(1,3).
\end{align*}
Taking $\epsilon\to 0$, we have that $\tfrac{d}{d\gamma}\nu(\gamma)$ may be arbitrarily close to zero around $\gamma^*=3/4$. We deduce that \cref{as:alg_gtrs} alone is not enough to upper bound $\tfrac{d}{d\gamma}\nu(\gamma)$ over $\inter(\Gamma)$.\ije{\hfill\proofbox}{}
\end{example}

\begin{lemma}\label{lem:coherence}
Suppose \cref{as:alg_gtrs} holds and that
\begin{align*}
\delta \coloneqq \norm{A_0A(\hat\gamma)^{-1}b_1 - A_1A(\hat\gamma)^{-1}b_0} >0.
\end{align*}
Then, $\frac{d}{d\gamma}\nu(\gamma) \leq - \delta^2\xi^2/(4\zeta^3)$ for any $\gamma\in\inter(\Gamma)$. In particular, $\abs{\nu(\gamma)}\leq \epsilon/(2\zeta)$ for an interval of length at most $4\epsilon\zeta^2/(\delta^2\xi^2)$.
\end{lemma}
\begin{proof}
For convenience, let $P\coloneqq A(\hat\gamma)^{-1/2}$ and $\Delta\coloneqq A_0P^2b_1 - A_1P^2b_0$ so that $\delta=\norm{\Delta}$.
By \cref{lem:nu_derivatives},
\begin{align*}
\frac{d}{d\gamma}\nu(\gamma) &= -2 \Delta^\top (A(\gamma)P^2 A(\gamma)P^2A(\gamma))^{-1}\Delta.
\end{align*}
\cref{as:alg_gtrs} implies $A(\hat\gamma)\succeq \xi I$, and so $P^2\preceq (1/\xi) I$. Moreover, by \cref{rem:regularity} we have $A(\gamma)\leq 2\zeta I$ $\forall \gamma \in\inter(\Gamma)$ and hence 
$A(\gamma)P^2 A(\gamma)P^2 A(\gamma) \preceq 8\zeta^3\xi^{-2} I$.
We conclude,
\begin{align*}
\frac{d}{d\gamma}\nu(\gamma) &\leq -\frac{\delta^2\xi^2}{4\zeta^3}.\qedhere
\end{align*}
\end{proof}

\begin{remark}\label{rem:no_case_2}
As in the proof of \cref{prop:nu_hat_gamma_positive_correctness}, we will assume that Line 1.(b) of \texttt{CRLeft} (\cref{alg:construct_reform_left}) succeeds in every iteration.
Suppose that \texttt{CRLeft} outputs \texttt{"maybe regular"} on iteration $t$. Recall that in this case we have $\nu(\gamma_t),\nu(\gamma')\in[\pm\epsilon/2\zeta]$ and $\mu_t \geq \mu^*/2$. By construction, $\gamma' = \gamma_t - \mu_t/4$. By \cref{lem:coherence} we deduce that the coherence parameter $\delta$ is bounded by
\begin{align*}
\delta \leq \frac{2\sqrt{2}\zeta}{\xi} \sqrt{\frac{\epsilon}{\mu^*}}.
\end{align*}
Momentarily treating $\xi,\zeta$ as constant, we deduce that \texttt{CRLeft} can only output \texttt{"maybe regular"} if the coherence parameter is sufficiently small, i.e., $\delta = O(\sqrt{\epsilon/\mu^*})$ (assuming that line 1.(b) succeeds in every iteration).\ije{\hfill\proofbox}{}
\end{remark}

\section{Numerical Experiments}
\label{sec:numerical}
In this section, we study the numerical performance of 
our approach (\cref{sec:constructing_sc_reform}) for solving the GTRS.
We compare our proposed approach with other algorithms \cite{wang2020generalized,jiang2019novel, adachi2019eigenvalue,benTal2014hidden} suggested in the literature.
In the following, we will refer to our algorithm as \texttt{WLK21} and the algorithms in \cite{wang2020generalized,jiang2019novel, adachi2019eigenvalue,benTal2014hidden} as \texttt{AN19}, \texttt{BTH14}, \texttt{JL19}, and \texttt{WK20} respectively.
Recall that \texttt{WK20} \cite{wang2020generalized} builds a convex reformulation of the GTRS (see \cref{rem:gtrs_convex_reform}) and applies Nesterov's accelerated gradient descent method.
\texttt{JL19} \cite{jiang2019novel} builds the same convex reformulation and applies a saddle-point-based first-order algorithm to solve it.
\texttt{AN19} \cite{adachi2019eigenvalue} 
computes the minimum generalized eigenvalue (and an associated eigenvector) of an indefinite $(2n+1)\times (2n+1)$ matrix pencil and recovers $\gamma^*$ and $x^*$ from these quantities.
\texttt{BTH14} \cite{benTal2014hidden} notes that the SDP relaxation of \eqref{eq:gtrs} (which is known to be exact) can be reformulated as a second-order cone program (SOCP) after computing an appropriate diagonalizing basis. The corresponding SOCP reformulation can then be solved via interior-point method solvers such as MOSEK.

In our experiments, we have implemented slight modifications to \texttt{WK20}, \texttt{WLK21}, \texttt{JL19}, and \texttt{AN19}.
First, we have replaced the eigenvalue calls within \texttt{WK20} and \texttt{WLK21} with generalized eigenvalue calls. Indeed, in both algorithms a series of eigenvalue calls are used to simulate a single generalized eigenvalue call. While the theoretical analysis using eigenvalue calls is simpler, the practical running time using generalized eigenvalue calls is faster due to the availability of efficient generalized eigenvalue solvers.
Second, in view of practical applications where $\epsilon$-feasibility may be unacceptable or undesirable, we also implement a ``rounding'' step at the ends of \texttt{WLK21}, \texttt{WK20}, and \texttt{JL19} to ensure feasibility, i.e., $q_1(\tilde x)\leq 0$.
As suggested in \cite{adachi2019eigenvalue}, \texttt{AN19} implements a Newton refinement process to ensure $q_1(\tilde x) \leq 0$.
The feasibility in \texttt{BTH14} depends on MOSEK and is often slightly violated.
Further implementation details are described in \cref{subsec:implementation}.

All experiments were performed in MATLAB R2021a and MOSEK 9.3.6 on a machine with an AMD Opteron 4184 processor and 70GB of RAM.

\subsection{Implementation}
\label{subsec:implementation}
We discuss some implementation details.

\paragraph{Eigenvalue solvers}
We replace \texttt{ApproxGammaLeft} (\cref{alg:ApproxGamma}) of \texttt{CRLeft} (\cref{alg:construct_reform_left}) using a generalized eigenvalue solver as follows.
Recall that \texttt{Approx\-Gamma\-Left} finds $\gamma_t\leq \hat \gamma$ and unit vector $v_t\in\R^n$ such that $\mu_t/2 \leq \mu(\gamma_t) \leq v_t^\top A(\gamma_t)v_t \leq \mu_t$.
We can achieve the same guarantee using a generalized eigenvalue solver: Approximate the minimum generalized eigenvalue $\lambda_t$ of $-A_1 v_t = \lambda_t (A(\hat\gamma) - \tfrac{3\mu_t}{4} I) v_t$ to some tolerance $\epsilon$ and set $\gamma_t = \hat\gamma + \tfrac{1}{\lambda_t}$. 
Then, as long as $\epsilon>0$ is small enough, we can show that $\gamma_t, v_t$ satisfy the same guarantees as \texttt{ApproxGammaLeft}.
Detailed proofs can be found in \cref{app:approxGamma}.
In our implementations, we use the generalized eigenvalue solver \texttt{eigifp}~\cite{golub2002inverse} for \texttt{WLK21}, \texttt{WK20} and \texttt{JL19}.
In contrast, as \texttt{AN19} requires the minimum eigenvalue to an \emph{indefinite} matrix pencil, we use the generalized eigenvalue solver \texttt{eigs} for \texttt{AN19}.

\paragraph{Rounding}
\label{subsubsec:rounding}
At the end of \texttt{WLK21}, \texttt{WK20} and \texttt{JL19}, we implement the following rounding procedure. 
Given the output $\bar x$ of one of these algorithms, we will construct $\tilde x \coloneqq \bar x + \delta$ where $\delta = \alpha v$. The direction $v$ is picked so that $x^\top A_1 x$ is either positive or negative depending on the sign of $q_1(\bar x)$. Then, we pick $\alpha$ by solving the quadratic equation $q_1(\bar x + \alpha v) = 0$.
For \texttt{WK20} and \texttt{JL19}, we may set $v$ to be an approximate eigenvector of $\gamma_-$ or $\gamma_+$ as we have already computed these quantities while constructing the convex reformulation.
For \texttt{WLK21}, we compute an (inaccurate) eigenvalue corresponding to either $\lambda_{\min}(A_1)$ or $\lambda_{\max}(A_1)$.

\subsection{Random instances}
We evaluate the numerical performance of the different algorithms on random instances with dimension $n$, number of nonzero entries $N\approx \bar N$, regularity $\mu^*\approx \bar\mu^*$, and $\xi = 0.1$.
Our random generation process is similar to that of~\cite{adachi2019eigenvalue} and allows us to generate instances with known optimizers.

First, sample a sparse symmetric matrix $A(\hat\gamma)$ using the MATLAB command \texttt{sprandsym(n,N/(n*n))}. This matrix is then scaled so that $0 \prec \xi I \preceq A(\hat\gamma) \preceq (1+\xi) I$. 
We generate $A_0$ using the same function call and scale it so that $\norm{A_0} \leq 1$. 
We then set $\hat\gamma \coloneqq \lambda_{\max}(A(\hat\gamma) - A_0)$ and $A_1 \coloneqq (A(\hat\gamma) - A_0) / \hat\gamma$. 
We sample $b_0$ and $b_1$ uniformly from the unit sphere.

We have the option to choose $\gamma^*$ to lie to either the left or right of $\hat\gamma$. In the former case, we set $\gamma^*\coloneqq \hat \gamma + 1 / \lambda_{\min}(-A_1, A(\hat\gamma) - \bar\mu I)$. In the latter, we set $\gamma^*\coloneqq \hat \gamma - 1 / \lambda_{\min}(A_1, A(\hat\gamma) - \bar\mu I)$.
To ensure that $\gamma^*$ is indeed the dual optimizer, we set $c_0 =0$ and $c_1$ such that $\nu(\gamma^*) = 0$.
The exact optimizer is then given by $x^* \coloneqq -A(\gamma^*)^{-1} b(\gamma^*)$. 
Finally, we normalize $b_0, b_1, c_1$ and $x^*$ to ensure \cref{as:alg_gtrs}.

To summarize, the output of this method is a random GTRS instance satisfying \cref{as:alg_gtrs} with $N\approx \bar N$, $\mu^*\approx \bar \mu^*$ and known $\Opt$ and $x^*$ (up to machine precision).

\subsection{Experimental setup}
\label{subsec:setup}
The numerical experiments were performed with $n \in \set{10^{3}, 10^{4}, 10^{5}}$, $\bar N \in \set{10 n, 100 n}$ and $\bar \mu^* \in \set{10^{-2}, 10^{-4}, 10^{-6}}$.
We generated 100 random instances for $n = 10^3$ and $10^4$ and five random instances for $n = 10^5$ due to large running times.
\texttt{BTH14} was only reported for $n = 10^3$ as for $n \geq 10^4$ it was unable to return a solution within five times the average running time of \texttt{WLK21} or \texttt{WK20}.
The dominant cost in \texttt{BTH14} for \eqref{eq:gtrs} is in computing the diagonalizing basis, which requires computing a full set of generalized eigenvalues and is unlikely to scale favorably with $n$ and $N$.
\texttt{AN19} was not reported for $n = 10^5$
because of numerical issues and large running times associated with \texttt{eigs} applied to the \emph{indefinite} generalized eigenvalue problem.

For each algorithm and each random instance, we record the error,
\begin{align*}
        \texttt{Error} &=  q_0(\tilde{x}) - \Opt,
\end{align*} 
of the output. For the three ``convex-reformulation and gradient-descent'' algorithms \texttt{WLK21}, \texttt{WK20}, and \texttt{JL19}, we additionally record the error \emph{within} the corresponding convex reformulations, i.e.,
\begin{align*}
        \texttt{ErrorCR} &= \max\left(q(\gamma^{(1)},\bar x), q(\gamma^{(2)},\bar x)\right) - \Opt,\quad \text{for } \texttt{WLK21},\quad\text{and}\\
        \texttt{ErrorCR} &= \max\left(q(\gamma_-,\bar x), q(\gamma_+,\bar x)\right) - \Opt,\quad \text{for } \texttt{WK20}\text{ and }\texttt{JL19}.
\end{align*}
See \eqref{eq:opt_equals_inf_sup} and \cref{thm:alg_gtrs_correctness} for definitions of $\gamma_-, \gamma_+$, $\gamma^{(1)}$ and $\gamma^{(2)}$.
Here, $\bar x$ is an iterate within the gradient descent method for the corresponding convex reformulation and $\tilde x$ is a ``rounded'' solution satisfying $q_1(\tilde x) \leq 0$.

\subsection{Results}
Our numerical results are illustrated in
\cref{fig:plot_1e3,fig:plot_1e4,fig:plot_1e5} which display \texttt{ErrorCR} for \texttt{WLK21}, \texttt{WK20}, and \texttt{JL19} and \texttt{Error} for \texttt{AN19} and \texttt{BTH14} over time (in seconds) for each $n\in\set{10^3, 10^4,10^5}$, respectively.
Tables containing detailed statistics are given in \cref{sec:tables}. We make a number of observations:

\begin{figure}[htbp]
        \centering
        \includegraphics[width=\textwidth]{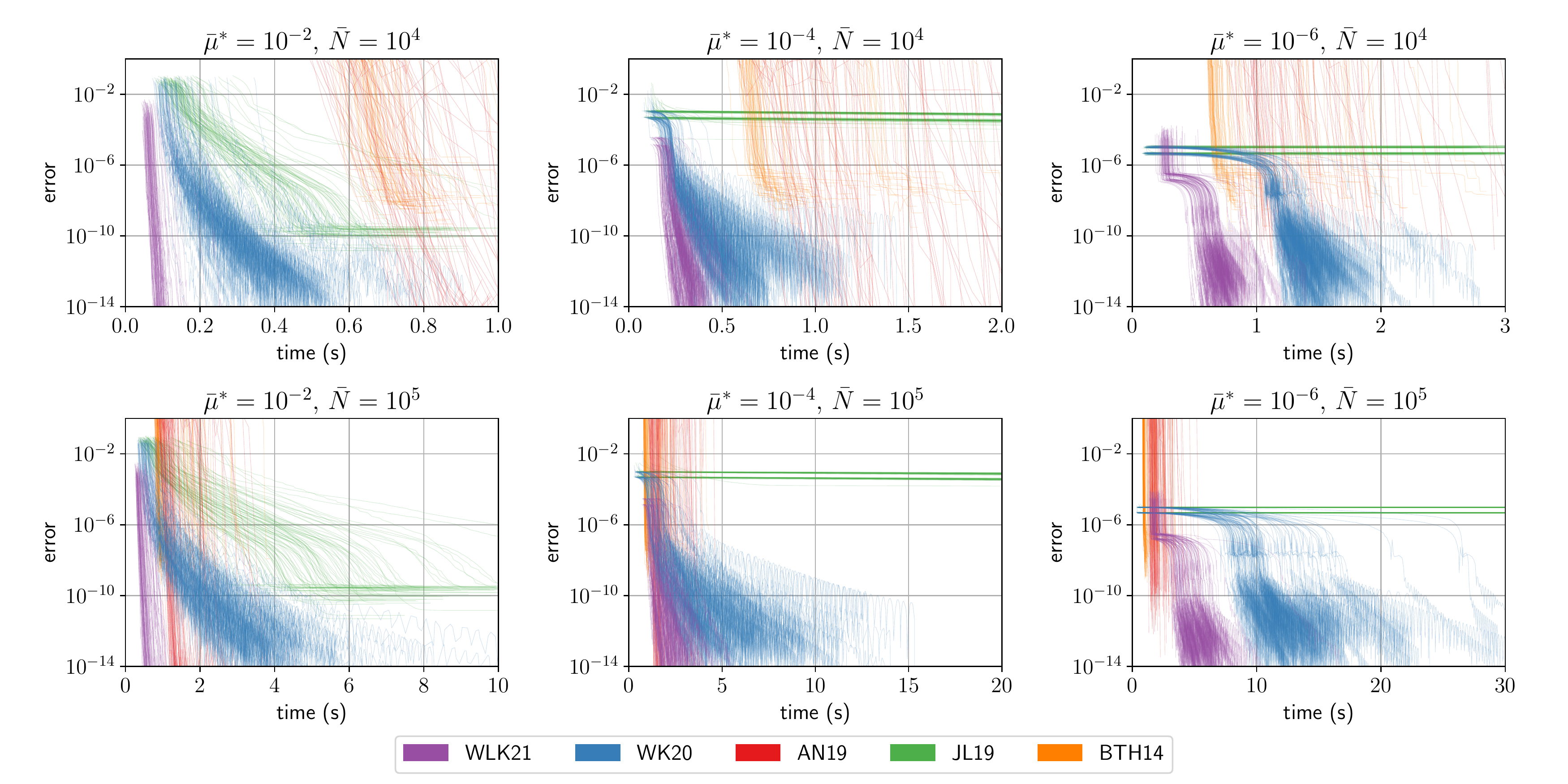}
        \caption{Comparison of algorithms for $n=10^3$.}
        \label{fig:plot_1e3}
\end{figure}

\begin{figure}[htbp]
        \centering
        \includegraphics[width=\textwidth]{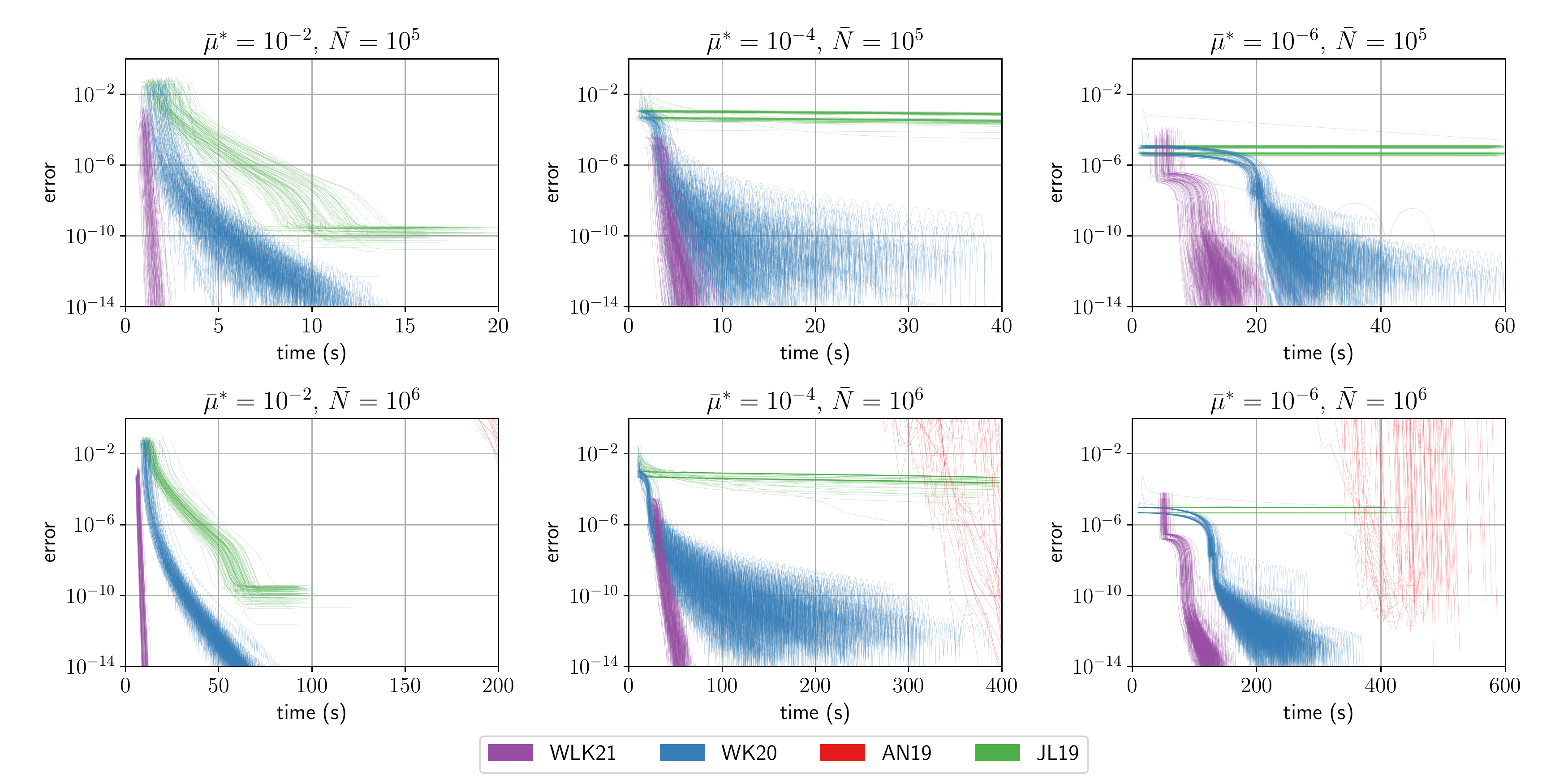}
        \caption{Comparison of algorithms for $n=10^4$.}
        \label{fig:plot_1e4}
\end{figure}

\begin{figure}[htbp]
        \centering
        \includegraphics[width=\textwidth]{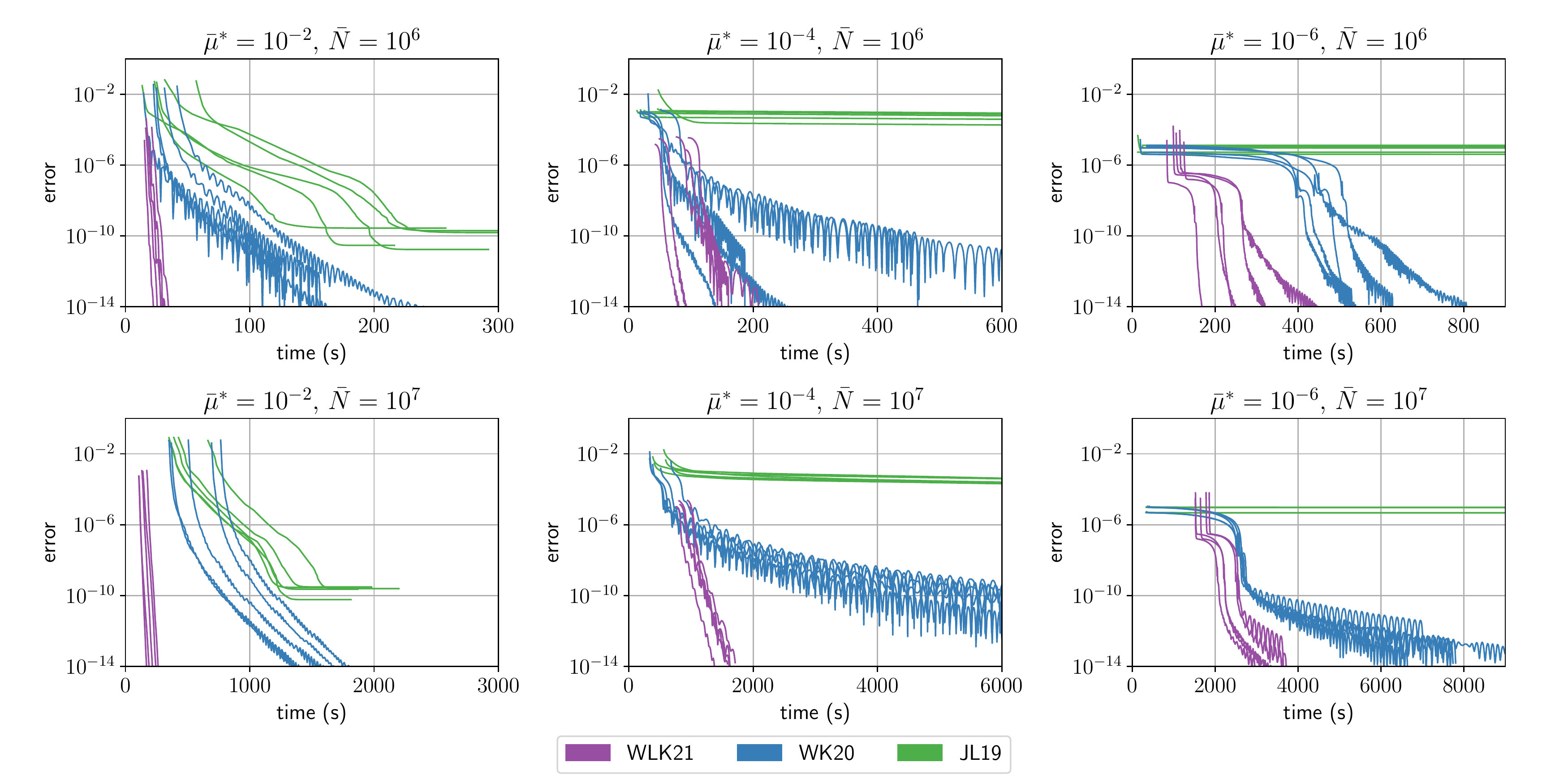}
        \caption{Comparison of algorithms for $n=10^5$.}
        \label{fig:plot_1e5}
\end{figure}

\begin{itemize}
        \item The lines plotted in \cref{fig:plot_1e3,fig:plot_1e4,fig:plot_1e5} begin after time zero. For \texttt{WLK21}, \texttt{WK20}, and \texttt{JL19} this gap corresponds to the time required to construct the corresponding convex reformulations of \eqref{eq:gtrs}. For \texttt{AN19}, this corresponds to the time required to compute $x(\hat\gamma)$ exactly, which is required to set up the appropriate $(2n+1)\times(2n+1)$ generalized eigenvalue problem~\cite{adachi2019eigenvalue}. 
For \texttt{BTH14}, this gap corresponds to the time required to compute a diagonalizing basis of \eqref{eq:gtrs}.
        \item \texttt{WLK21} constructs its reformulation faster than \texttt{WK20} and \texttt{JL19} when $\bar\mu^* = 10^{-2}$. The situation is reversed for $\bar\mu^* \in\set{10^{-4},10^{-6}}$. Nevertheless, \texttt{WLK21} outperforms both \texttt{WK20} and \texttt{JL19} due to its significantly improved performance in solving the resulting convex reformulation. See \cref{sec:tables}.

        \item As expected from \cref{thm:overall_time},
        \texttt{WLK21} exhibits a \emph{linear} convergence rate in terms of $\epsilon$. This is most apparent in the plots corresponding to $\bar\mu^* =10^{-2}$ and $\bar \mu^* = 10^{-4}$.

        \item Although the convergence guarantees established for \texttt{WK20} \cite{wang2020generalized} and \texttt{JL19} \cite{jiang2019novel} do not depend on $\mu^*$, our results show empirically that these algorithms in fact perform better when $\mu^*$ is large.
        The degree to which the running times of these algorithms vary with $\mu^*$ is less than that of \texttt{WLK21}.

        \item The convergence rates of \texttt{AN19} and \texttt{BTH14} do not vary significantly with either $N$ or $\mu^*$, but they exhibit heavy dependence on $n$. Specifically, the convergence rate of \texttt{AN19} empirically varies in $n$ as $\approx n^2$. This is consistent with the results reported in \cite{adachi2019eigenvalue}. Similarly, due to the complete eigenbasis computation embedded in \texttt{BTH14}, we expect \texttt{BTH14} to vary in $n$ as $\approx n^3$. Thus, as can be seen in \cref{fig:plot_1e3,fig:plot_1e4,fig:plot_1e5}, although \texttt{AN19} outperforms \texttt{WLK21} and \texttt{WK20} for $(n,\bar N,\bar\mu^*) = (10^3, 10^5, 10^{-6})$, \texttt{AN19} and \texttt{BTH14} become impractical for $n = 10^4$ and $n = 10^5$.

        \item The saddle-point based first-order algorithm employed in \texttt{JL19} is unable to decrease the error below $\approx 10^{-4}$ for $\bar\mu^* = 10^{-4}$ and $\bar\mu^* = 10^{-6}$.
\end{itemize}

{
\bibliographystyle{plainnat}

\begin{thebibliography}{34}
\providecommand{\natexlab}[1]{#1}
\providecommand{\url}[1]{\texttt{#1}}
\expandafter\ifx\csname urlstyle\endcsname\relax
  \providecommand{\doi}[1]{doi: #1}\else
  \providecommand{\doi}{doi: \begingroup \urlstyle{rm}\Url}\fi

\bibitem[Adachi and Nakatsukasa(2019)]{adachi2019eigenvalue}
S.~Adachi and Y.~Nakatsukasa.
\newblock Eigenvalue-based algorithm and analysis for nonconvex {QCQP} with one
  constraint.
\newblock \emph{{Math.\ Program.}}, 173:\penalty0 79--116, 2019.

\bibitem[Ben-Tal and {den Hertog}(2014)]{benTal2014hidden}
A.~Ben-Tal and D.~{den Hertog}.
\newblock Hidden conic quadratic representation of some nonconvex quadratic
  optimization problems.
\newblock \emph{{Math.\ Program.}}, 143:\penalty0 1--29, 2014.

\bibitem[Ben-Tal and Nemirovski(2001)]{benTal2001lectures}
A.~Ben-Tal and A.~Nemirovski.
\newblock \emph{Lectures on Modern Convex Optimization}, volume~2 of
  \emph{MPS-SIAM Ser.\ Optim.}
\newblock SIAM, 2001.

\bibitem[Ben-Tal and Teboulle(1996)]{benTal1996hidden}
A.~Ben-Tal and M.~Teboulle.
\newblock Hidden convexity in some nonconvex quadratically constrained
  quadratic programming.
\newblock \emph{{Math.\ Program.}}, 72:\penalty0 51--63, 1996.

\bibitem[Carmon and Duchi(2018)]{carmon2018analysis}
Y.~Carmon and J.~C. Duchi.
\newblock Analysis of {Krylov} subspace solutions of regularized nonconvex
  quadratic problems.
\newblock \emph{{arXiv preprint}}, 1806.09222, 2018.

\bibitem[Conn et~al.(2000)Conn, Gould, and Toint]{conn2000trust}
A.~R. Conn, N.~I. Gould, and P.~L. Toint.
\newblock \emph{Trust Region Methods}, volume~1 of \emph{MPS-SIAM Ser.\ Optim.}
\newblock SIAM, 2000.

\bibitem[Ekeland and Temam(1999)]{ekeland1999convex}
I.~Ekeland and R.~Temam.
\newblock \emph{Convex Analysis and Variational Problems}, volume~28 of
  \emph{Classics Appl.\ Math.}
\newblock SIAM, 1999.

\bibitem[Fallahi et~al.(2018)Fallahi, Salahi, and
  Terlaky]{fallahi2018minimizing}
S.~Fallahi, M.~Salahi, and T.~Terlaky.
\newblock Minimizing an indefinite quadratic function subject to a single
  indefinite quadratic constraint.
\newblock \emph{Optimization}, 67\penalty0 (1):\penalty0 55--65, 2018.

\bibitem[Feng et~al.(2012)Feng, Xuan, Sheu, and Xia]{feng2012duality}
J.M. Feng, G.X. Xuan, R.L. Sheu, and Y.~Xia.
\newblock Duality and solutions for quadratic programming over single
  non-homogeneous quadratic constraint.
\newblock \emph{{J.\ Global Optim.}}, 54\penalty0 (2):\penalty0 275--293, 2012.

\bibitem[Fortin and Wolkowicz(2004)]{fortin2004trust}
C.~Fortin and H.~Wolkowicz.
\newblock The {Trust Region Subproblem} and semidefinite programming.
\newblock \emph{Optim. Methods and Softw.}, 19\penalty0 (1):\penalty0 41--67,
  2004.

\bibitem[Fradkov and Yakubovich(1979)]{fradkov1979s-procedure}
A.~L. Fradkov and V.~A. Yakubovich.
\newblock The {S}-procedure and duality relations in nonconvex problems of
  quadratic programming.
\newblock \emph{Vestnik Leningrad Univ.\ Math.}, 6:\penalty0 101--109, 1979.

\bibitem[Golub and Ye(2002)]{golub2002inverse}
G.~H. Golub and Q.~Ye.
\newblock An inverse free preconditioned krylov subspace method for symmetric
  generalized eigenvalue problems.
\newblock \emph{SIAM Journal on Scientific Computing}, 24\penalty0
  (1):\penalty0 312--334, 2002.

\bibitem[Gould et~al.(1999)Gould, Lucidi, Roma, and Toint]{gould1999solving}
N.~I.~M. Gould, S.~Lucidi, M.~Roma, and P.~L. Toint.
\newblock Solving the {Trust-Region Subproblem} using the {L}anczos method.
\newblock \emph{{SIAM J.\ Optim.}}, 9\penalty0 (2):\penalty0 504--525, 1999.

\bibitem[Hazan and Koren(2016)]{hazan2016linear}
E.~Hazan and T.~Koren.
\newblock A linear-time algorithm for trust region problems.
\newblock \emph{{Math.\ Program.}}, 158:\penalty0 363--381, 2016.

\bibitem[Hmam(2010)]{hmam2010quadratic}
H.~Hmam.
\newblock Quadratic optimisation with one quadratic equality constraint.
\newblock Technical report, Defence Science and Technology Organisation
  Edinburgh (Australia) Electronic Warfare and Radar Division, 2010.

\bibitem[Ho-Nguyen and K{\i}l{\i}n\c{c}-Karzan(2017)]{hoNguyen2017second}
N.~Ho-Nguyen and F.~K{\i}l{\i}n\c{c}-Karzan.
\newblock A second-order cone based approach for solving the {Trust Region
  Subproblem} and its variants.
\newblock \emph{{SIAM J.\ Optim.}}, 27\penalty0 (3):\penalty0 1485--1512, 2017.

\bibitem[Huang and Sidiropoulos(2016)]{huang2016consensus}
K.~Huang and N.~D. Sidiropoulos.
\newblock Consensus-{ADMM} for general quadratically constrained quadratic
  programming.
\newblock \emph{IEEE Transactions on Signal Processing}, 64\penalty0
  (20):\penalty0 5297--5310, 2016.

\bibitem[Jiang and Li(2016)]{jiang2016simultaneous}
R.~Jiang and D.~Li.
\newblock Simultaneous diagonalization of matrices and its applications in
  quadratically constrained quadratic programming.
\newblock \emph{{SIAM J.\ Optim.}}, 26\penalty0 (3):\penalty0 1649--1668, 2016.

\bibitem[Jiang and Li(2019)]{jiang2019novel}
R.~Jiang and D.~Li.
\newblock Novel reformulations and efficient algorithms for the {Generalized
  Trust Region Subproblem}.
\newblock \emph{{SIAM J.\ Optim.}}, 29\penalty0 (2):\penalty0 1603--1633, 2019.

\bibitem[Jiang and Li(2020)]{jiang2020linear}
R.~Jiang and D.~Li.
\newblock A linear-time algorithm for generalized trust region problems.
\newblock \emph{{SIAM J.\ Optim.}}, 30\penalty0 (1):\penalty0 915--932, 2020.

\bibitem[Jiang et~al.(2018)Jiang, Li, and Wu]{jiang2018socp}
R.~Jiang, D.~Li, and B.~Wu.
\newblock {SOCP} reformulation for the {Generalized Trust Region Subproblem}
  via a canonical form of two symmetric matrices.
\newblock \emph{{Math.\ Program.}}, 169:\penalty0 531--563, 2018.

\bibitem[Karmarkar et~al.(1991)Karmarkar, Resende, and
  Ramakrishnan]{karmarkar1991interior}
N.~Karmarkar, M.~G. Resende, and K.~G. Ramakrishnan.
\newblock An interior point algorithm to solve computationally difficult set
  covering problems.
\newblock \emph{{Math.\ Program.}}, 52:\penalty0 597--618, 1991.

\bibitem[Kuczynski and Wozniakowski(1992)]{KuczynskiWozniakowski1992estimating}
J.~Kuczynski and H.~Wozniakowski.
\newblock Estimating the largest eigenvalue by the power and {Lanczos}
  algorithms with a random start.
\newblock \emph{{SIAM J.\ Matrix Anal.\ Appl.}}, 13\penalty0 (4):\penalty0
  1094--1122, 1992.

\bibitem[Locatelli(2015)]{locatelli2015some}
M.~Locatelli.
\newblock Some results for quadratic problems with one or two quadratic
  constraints.
\newblock \emph{{Oper.\ Res.\ Lett.}}, 43\penalty0 (2):\penalty0 126--131,
  2015.

\bibitem[Mor{\'e}(1993)]{more1993generalizations}
J.~J. Mor{\'e}.
\newblock Generalizations of the trust region problem.
\newblock \emph{Optim. methods and Softw.}, 2\penalty0 (3-4):\penalty0
  189--209, 1993.

\bibitem[Mor{\'e} and Sorensen(1983)]{more1983computing}
J.~J. Mor{\'e} and D.~C. Sorensen.
\newblock Computing a trust region step.
\newblock \emph{SIAM J. on Sci. and Stat. Comput.}, 4\penalty0 (3):\penalty0
  553--572, 1983.

\bibitem[Nesterov(2018)]{nesterov2018lectures}
Y.~Nesterov.
\newblock \emph{Lectures on convex optimization}.
\newblock Number 137 in Springer Optim. and its Appl. Springer, 2 edition,
  2018.

\bibitem[Pardalos et~al.(1991)Pardalos, Ye, and Han]{pardalos1991algorithms}
P.~M. Pardalos, Y.~Ye, and CG~Han.
\newblock Algorithms for the solution of quadratic knapsack problems.
\newblock \emph{{Linear Algebra Appl.}}, 152:\penalty0 69--91, 1991.

\bibitem[P{\'o}lik and Terlaky(2007)]{polik2007survey}
I.~P{\'o}lik and T.~Terlaky.
\newblock A survey of the {S}-lemma.
\newblock \emph{{SIAM Rev.}}, 49\penalty0 (3):\penalty0 371--418, 2007.

\bibitem[Pong and Wolkowicz(2014)]{PongWolkowicz2014}
T.K. Pong and H.~Wolkowicz.
\newblock The generalized trust region subproblem.
\newblock \emph{Computational Optimization and Applications}, 58\penalty0
  (2):\penalty0 273--322, 2014.

\bibitem[Salahi and Fallahi(2016)]{salahi2016trust}
M.~Salahi and S.~Fallahi.
\newblock Trust region subproblem with an additional linear inequality
  constraint.
\newblock \emph{Optim. Lett.}, 10\penalty0 (4):\penalty0 821--832, 2016.

\bibitem[Stern and Wolkowicz(1995)]{stern1995indefinite}
R.~J. Stern and H.~Wolkowicz.
\newblock Indefinite trust region subproblems and nonsymmetric eigenvalue
  perturbations.
\newblock \emph{{SIAM J.\ Optim.}}, 5\penalty0 (2):\penalty0 286--313, 1995.

\bibitem[Wang and Jiang(2021)]{wang2021new}
A.~L. Wang and R.~Jiang.
\newblock New notions of simultaneous diagonalizability of quadratic forms with
  applications to {QCQPs}.
\newblock \emph{{arXiv preprint}}, 2101.12141, 2021.

\bibitem[Wang and K{\i}l{\i}n\c{c}-Karzan(2020)]{wang2020generalized}
A.~L. Wang and F.~K{\i}l{\i}n\c{c}-Karzan.
\newblock The generalized trust region subproblem: solution complexity and
  convex hull results.
\newblock \emph{{Math.\ Program.}}, 2020.
\newblock \doi{10.1007/s10107-020-01560-8}.
\newblock Forthcoming.

\end{thebibliography}

}

\clearpage
\begin{appendix}

\section{Useful lemmas regarding quadratic functions}
\label{app:quadratic}

The following two basic bounds will be useful in our error analysis.
\begin{lemma}
\label{lem:quadratic_error}
Let $q(x) = x^\top Ax + 2b^\top x + c$ for $A\in\S^n$, $b\in\R^n$, and $c\in\R$. Then, for all $x,y\in\R^n$,
$\abs{q(x) - q(y)} \leq \norm{A}\norm{y-x}^2 + 2\left(\norm{A}\norm{x}+\norm{b}\right)\norm{y-x}$.
In particular, if $\norm{A},\norm{b}\leq 1$,
$\norm{x}\leq \rho$ and $\norm{x-y}\leq \delta$ for some $\delta\leq 1\leq \rho$, then
$\abs{q(x) - q(y)} \leq 5\delta\rho$.
\end{lemma}
\begin{proof}
Writing $y = (y - x) + x$ and expanding the formula for $q(y)$, we obtain
\begin{align*}
q(y) &= (y-x)^\top A (y-x) + 2x^\top A (y-x) + x^\top A x + 2b^\top(y-x) + 2b^\top x + c\\
&= q(x) + \left((y-x)^\top A (y-x) + 2\ip{Ax + b, y-x}\right).\qedhere
\end{align*}
\end{proof}

\begin{lemma}\label{lem:quadraticRootsBound}
Let $\alpha,\beta,\gamma\in\R$ where $\alpha \neq 0$ and $\gamma/\alpha \leq 0$. Then the roots of $\alpha z^2 + 2\beta z + \gamma=0$ satisfy $\abs{z} \leq 2\abs{\frac{\beta}{\alpha}} + \sqrt{\frac{-\gamma}{\alpha}}$.
\end{lemma}
\begin{proof}
Let $\set{z_-,z_+}$ denote the roots (possibly with multiplicity). We bound
\begin{align*}
\set{z_-,\, z_+} &= \set{-\tfrac{\beta}{\alpha} \pm \sqrt{\left(\tfrac{\beta}{\alpha}\right)^2 - \tfrac{\gamma}{\alpha}} }\\
&\subseteq \left[-\tfrac{\beta}{\alpha} - \left(\abs{\tfrac{\beta}{\alpha}} + \sqrt{\tfrac{-\gamma}{\alpha}}\right),\, -\tfrac{\beta}{\alpha} + \left(\abs{\tfrac{\beta}{\alpha}} + \sqrt{\tfrac{-\gamma}{\alpha}}\right)\right] \\
& \subseteq \left[-\left(2\abs{\tfrac{\beta}{\alpha}} + \sqrt{\tfrac{-\gamma}{\alpha}}\right),\,  \left(2\abs{\tfrac{\beta}{\alpha}} + \sqrt{\tfrac{-\gamma}{\alpha}}\right)\right].\qedhere
\end{align*}
\end{proof}

\section{Useful procedures}
\label{app:useful_procedures}

This appendix contains running time guarantees for well-known algorithms that we will utilize as building blocks in \cref{alg:construct_reform_gtrs}.

\subsection{The Lanczos method}
The following lemma characterizes the running time for approximating the minimum eigenvalue of a symmetric matrix.
\begin{lemma}[\cite{KuczynskiWozniakowski1992estimating}]
	\label{lemma:ApproxEig}
	There exists an algorithm, $\textup{ApproxEig}(A, \rho, \delta, p)$, 
	which given a symmetric matrix $A\in\S^n$, 
	$\rho$ such that $\norm{A}_2\leq\rho$, and parameters $\delta, p>0$, 
	will, with probability at least $1-p$, return a unit vector 
	$x\in\R^n$ such that $x^\top Ax\leq \lambda_{\min}(A)+\delta$. 
	This algorithm runs in time
	\begin{align*}
		O\left( \tfrac{N\sqrt{\rho}}{\sqrt{\delta}} \log\left(\tfrac{n}{p}\right) \right),
	\end{align*}
	where $N$ is the number of nonzero entries in $A$.
\end{lemma}

\subsection{ApproxGamma}
The following algorithm extends \cite[Algorithm 2]{wang2020generalized} to find a $\gamma\leq \hat\gamma$ such that $\mu(\gamma)$ falls in a prescribed range. An analogous algorithm can be used to find a $\gamma\geq\hat\gamma$ such that $\mu(\gamma)$ falls in a prescribed range.

\begin{algorithm}
	\caption{\texttt{ApproxGammaLeft}}
	\label{alg:ApproxGamma}
	Given $(A_0,A_1)$, $(\xi,\zeta,\hat\gamma)$,
	$p\in(0,1)$, and $\mu\in(0,\xi)$
	{\small\begin{enumerate}[topsep=0pt,itemsep=0pt,parsep=0pt]
		\item Set $\ell_1 = 0$, $r_1 = \hat\gamma$
		\item For $t=1,\dots,T = \ceil{\log\left(\tfrac{5\zeta}{\mu}\right)}$
		\begin{enumerate}
			\item $\gamma_t = (\ell_t+r_t)/2$
			\item Let $x_t = ApproxEig\left(A(\gamma_t), 2\zeta, \mu /8, p/T\right)$ and $\hat\mu_t = x_t^\top A(\gamma_t)x_t$
			\item If $\hat\mu_t > \mu$, set $\ell_{t+1} = \ell_t,~ r_{t+1} = \gamma_t$
			\item Else if $\hat\mu_t < \tfrac{5}{8}\mu$, set $\ell_{t+1} = \gamma_t,~ r_{t+1} = r_t$
			\item Else, output $\gamma_t$, $x_t$
		\end{enumerate}
	\end{enumerate}}
\end{algorithm}
\lemapproxgamma*
\begin{proof}
We condition on ApproxEig succeeding in each call. This happens with probability at least $1-p$.

Suppose \texttt{ApproxGammaLeft} outputs on iteration $t$. On this iteration, we have $\mu(\gamma_t) \geq \hat\mu_t - \mu/8 \geq \mu/2$.
Similarly note $x^\top A(\gamma_t)x = \hat\mu_t \leq \mu$.

Next, we show that \texttt{ApproxGammaLeft} is guaranteed to output within $T$ iterations. Suppose otherwise and consider the interval
\begin{align*}
\cI \coloneqq \set{\gamma\in\R_+:\, \begin{array}
	{l}
	\gamma\leq \hat\gamma\\
	\mu(\gamma) \in \left[\tfrac{5}{8}\mu, \tfrac{7}{8}\mu\right]
\end{array}}.
\end{align*}
Note that if $\gamma_t \in\cI$ for some $t$ then \texttt{ApproxGammaLeft} will output at step $t$. Indeed, at iteration $t$, we will have $\hat\mu_t \in \left[\mu(\gamma_t), \mu(\gamma_t) + \tfrac{\mu}{8}\right] \subseteq \left[\tfrac{5}{8}\mu, \mu\right]$. In particular, we deduce that $\gamma_t\notin \cI$ for any $t = 1,\dots, T$. Next, by construction, the interval $[\ell_t, r_t]$ contains $\cI$ for every $t$. On the other hand, $\abs{[\ell_T, r_T]} \leq 2^{-T}\zeta<\tfrac{\mu}{4}\leq \abs{\cI}$, a contradiction.

It remains to bound the running time of \texttt{ApproxGammaLeft}.
By \cref{lemma:ApproxEig}, each iteration of step 2.(b) runs in time
\begin{align*}
\tilde O\left(\tfrac{N\sqrt{\zeta}}{\sqrt{\mu}}\log\left(\tfrac{n}{p}\right)\right).
\end{align*}
Finally, note that the number of iterations of step 2 is bounded by $T = O\left(\log\left(\tfrac{\zeta}{\mu}\right)\right)$.
\end{proof}

\subsection{Conjugate gradient}
The following lemma characterizes the running time for approximately minimizing a strongly convex quadratic function using the conjugate gradient algorithm.
\begin{lemma}
	\label{lem:conjugate_grad}
	There exists an algorithm, $\textup{ConjGrad}(A, b, \rho,\mu,\delta)$, which given symmetric matrix $A\in\S^n$ with $\mu I \preceq A\preceq \rho I$ and $b\in\R^n$, returns $x\in\R^n$ such that $\norm{x + A^{-1}b} \leq \delta$. This algorithm runs in time
	\begin{align*}
	O\left(\tfrac{N\sqrt{\rho}}{\sqrt{\mu}}\log\left(\tfrac{\norm{b}}{\mu\delta}\right)\right).
	\end{align*}
\end{lemma}

\subsection{ApproxNu}
The following algorithm uses the conjugate gradient algorithm to approximate $\nu(\gamma)$ for a given value of $\gamma$. 

\begin{algorithm}
	\caption{\texttt{ApproxNu}}
	\label{alg:approxnu}
	Given $(A_0,A_1,b_0,c_0,c_1)$, $(\xi,\zeta,\hat\gamma)$ satisfying \cref{as:alg_gtrs}, $\gamma,\mu$ such that $\mu\in(0,1)$ and $A(\gamma)\succeq \mu I$, and $\delta>0$
	{\small
	\begin{itemize}
		\item Apply the conjugate gradient method to find $\tilde x$ such that $\norm{\tilde x - x(\gamma)} \leq \tfrac{\mu\delta}{10\zeta}$
		\item Return $\tilde x$, $q_1(\tilde x)$
	\end{itemize}
	}
\end{algorithm}

\lemapproxnu*
\begin{proof}
The running time follows from \cref{lem:conjugate_grad}. Note that \cref{as:alg_gtrs} and $A(\gamma)\succeq \mu I$ together imply $\norm{x(\gamma)}\leq \tfrac{2\zeta}{\mu}$. Then, from the definition of $\nu(\gamma)$ and $x(\gamma)$ and applying \cref{lem:quadratic_error}, we arrive at
\begin{align*}
\abs{q_1(\hat x) - \nu(\gamma)} &\leq 5\left(\tfrac{2\zeta}{\mu}\right)\left(\tfrac{\mu\delta}{10\zeta}\right)\leq \delta.\qedhere
\end{align*}
\end{proof}

\subsection{Nesterov's accelerated minimax scheme}
The following lemma characterizes the running time for finding an approximate optimizer of the maximum of two strongly convex smooth quadratic functions.
\begin{lemma}
\label{lem:acc_minimax}
There exists an algorithm, $\textup{AccMinimax}$, which given $A^{(1)},A^{(2)}\in\S^n$, $b^{(1)}, b^{(2)}\in\R^n$, $c^{(1)}, c^{(2)}\in\R$, and $(\mu,\rho,\delta)>0$
satisfying $\mu I \preceq A^{(i)}\preceq \rho I$ and
$\norm{b^{(i)}}\leq \rho$, will return $\bar x$ such that
\begin{align*}
\max_i\, \bar x^\top A^{(i)} \bar x + 2 b^{(i)\top}\bar x + c_i \leq \left(\min_{x\in\R^n}\max_i\, x^\top A^{(i)} x + 2 b^{(i)\top}x + c_i\right) + \delta,
\end{align*}
in time
\begin{align*}
O\left(\tfrac{N\sqrt{\rho}}{\sqrt{\mu}}\log\left(\tfrac{\rho}{\delta\mu}\right)\right).
\end{align*}
\end{lemma}
\begin{proof}
For notational convenience, define $q^{(i)}(x) \coloneqq x^\top A^{(i)} x + 2b^{(i)\top}x + c^{(i)}$ and $f(x)\coloneqq \max_i q^{(i)}(x)$.
We may take $x_0=0$ in \cite[Algorithm 2.3.12]{nesterov2018lectures} and bound
\begin{align*}
f(0) - \min_x f(x) &\leq f(0) - \max_i\min_x q^{(i)}(x)\\
&\leq \max_i \left(q^{(i)}(0) - \min_x q^{(i)}(x)\right)\\
&= \max_i\, b^{(i)\top} \left(A^{(i)}\right)^{-1} b^{(i)}\\
&\leq \tfrac{\rho^2}{\mu}.
\end{align*}
The running time then follows from \cite[Theorem 2.3.5]{nesterov2018lectures} and \cite[Lemma 14]{wang2020generalized}.
\end{proof}

 	\section{Deferred proofs from \cref{sec:regularity}}
\label{app:regularity}
\begin{lemma}
Suppose \cref{as:definiteness} holds. Then
\begin{align*}
\Opt = \inf_{x\in\R^n}\sup_{\gamma\in\Gamma}q(\gamma,x).
\end{align*}
\end{lemma}
\begin{proof}
($\geq$) Let $x\in\R^n$ such that $q_1(x)\leq 0$. Then, as $\Gamma\subseteq\R_+$, we have $q_0(x)\geq \sup_{\gamma\in\Gamma} q(\gamma, x)$. Taking the infimum in $x$ concludes this direction.

($\leq$) Let $x\in\R^n$. We split into three cases depending on the sign of $q_1(x)$.

If $q_1(x) = 0$, then $\Opt\leq q_0(x) = \sup_{\gamma\in\Gamma} q(\gamma,x)$.

Next, suppose $q_1(x) < 0 $ so that $\sup_{\gamma\in\Gamma} q(\gamma, x) = q(\gamma_-, x)$. If $\gamma_- = 0$, then again $\Opt \leq q_0(x) = \sup_{\gamma\in\Gamma} q(\gamma,x)$. On the other hand, if $\gamma_->0$, then $A(\gamma_-)\not\succeq 0$ and there exists nonzero $v\in\ker(A(\gamma_-))$. Without loss of generality, $\ip{v,b(\gamma_-)} \leq 0$. Let $\alpha>0$ such that $q_1(x+ \alpha v) = 0$ (this exists as $v^\top A_1 v = v^\top \frac{A(\bar\gamma) - A(\gamma_-)}{\bar \gamma - \gamma_-}v > 0$). We deduce $\Opt \leq q_0(x+\alpha v) = q(\gamma_-, x+\alpha v) \leq q(\gamma_-, x) = \sup_{\gamma\in\Gamma} q(\gamma,x)$.

Finally, suppose $q_1(x)>0$. If $\Gamma$ is unbounded, then $\sup_{\gamma\in\Gamma} q(\gamma,x) = + \infty$ and $\Opt \leq \sup_{\gamma\in\Gamma} q(\gamma,x)$. Else, we have that $A(\gamma_+)\not\succeq 0$ and there exists nonzero $v\in\ker(A(\gamma_+))$. An argument identical to the one in the previous paragraph shows $\Opt \leq \sup_{\gamma\in\Gamma} q(\gamma,x)$.

Taking the infimum over all $x\in\R^n$ completes the proof.
\end{proof} 	

\section{Deferred proofs from \cref{subsec:implementation}}
\label{app:approxGamma}

In this appendix, we motivate a generalized-eigenvalue-based replacement for \texttt{ApproxGammaLeft} (\cref{alg:ApproxGamma}) of \texttt{CRLeft} (\cref{alg:construct_reform_left}).
Given $\mu \in(0,\xi)$,
our goal is to compute $\gamma\leq\hat\gamma$ and $v$ such that $\mu/2 \leq \mu(\gamma) \leq v^\top A(\gamma)v \leq \mu$. We will do so by approximating the minimum eigenvalue $\tilde \lambda$ (and a corresponding eigenvector) for
\begin{align}
\label{eq:eigifp}
-A_1 v = \lambda \left(A(\hat\gamma) - \tfrac{3\mu}{4} I\right)v
\end{align}
and setting $\tilde \gamma \coloneqq \hat\gamma + \tfrac{1}{\tilde\lambda}$.
Note that defining $\gamma \coloneqq \hat\gamma + \tfrac{1}{\lambda}$, where $\lambda$ is the true minimum eigenvalue to \eqref{eq:eigifp}, gives
\begin{align*}
\mu(\gamma) &= \lambda_{\min}\left(A(\hat \gamma)- \tfrac{3\mu }{4} I + \tfrac{1}{\lambda} A_1\right) + 3\mu/4= 3\mu/4.
\end{align*}
In the following, we abbreviate $\hat A \coloneqq A(\hat \gamma) - \tfrac{3\mu}{4}I$.
As in \cref{lem:approx_gamma}, we will assume \cref{as:alg_gtrs} throughout this appendix.
We will take $\tilde\lambda, \tilde v$ to be the output of \texttt{eigifp} on the input $(-A_1, \hat A, \delta)$ where $\delta>0$ will be fixed later.

Recall~\cite{golub2002inverse} that $\tilde\lambda,\tilde v$ satisfies
\begin{align}
\label{eq:eigifp_perturbed}
(-A_1 + B) \tilde v = \tilde\lambda (\hat A + C)\tilde v
\end{align}
for some $\norm{B} \leq \delta\norm{A_1}$ and $\norm{C}\leq \delta\norm{\hat A}$. We will assume that $\tilde\lambda$ is in fact the \emph{minimum eigenvalue} of \eqref{eq:eigifp_perturbed}.

\begin{lemma}
\label{lem:add_to_mult}
Suppose $\abs{\lambda - \tilde\lambda} \leq \mu/5\zeta^2$, then $\mu(\tilde \gamma) \geq \mu/2$.
\end{lemma}
\begin{proof}
As $\mu(\tilde\gamma)$ is $1$-Lipschitz, it suffices to show that $\abs{\tilde\gamma-\gamma}\leq \mu/4$. Note that $\frac{1}{\lambda} = \gamma - \hat\gamma$ so that $\abs{\lambda}\geq 1/\zeta$. We deduce that $\abs{\tilde \lambda} \geq \abs{\lambda}- \abs{\lambda - \tilde\lambda}$. Combining,
\begin{align*}
\abs{\tilde\gamma-\gamma} &= \frac{\abs{\lambda - \tilde\lambda}}{\abs{\lambda}\abs{\tilde\lambda}} \leq \frac{\tfrac{\mu}{5\zeta^2}}{\left(\tfrac{1}{\zeta}\right)\left(\tfrac{1}{\zeta} - \tfrac{\mu}{5\zeta^2}\right)} \leq \mu/4.\qedhere
\end{align*}
\end{proof}

\begin{lemma}
\label{lem:lambda_additive_bound}
Suppose $\tilde\lambda$ is a minimum eigenvalue of \eqref{eq:eigifp_perturbed}
and $2\delta\zeta\leq \xi/8$.
Then,
\begin{align*}
\abs{\lambda -\tilde\lambda}\leq \delta\tfrac{72\zeta}{\xi^2}.
\end{align*}
\end{lemma}
\begin{proof}
Note that
\begin{align*}
\lambda &= \max\set{\lambda:\, -A_1 - \lambda \hat A\succeq 0},\quad\text{and}\quad
\tilde\lambda = \max\set{\tilde\lambda:\, (-A_1+B) - \tilde\lambda (\hat A+C)\succeq 0}.
\end{align*}

We compute
\begin{align*}
-A_1 - (\tilde\lambda - \alpha) \hat A &= (-A_1 + B) - \tilde \lambda (\hat A + C) - B + \tilde \lambda C +\alpha \hat A\\
&\succeq -\delta(1 + 2\zeta|\tilde\lambda|) +\alpha\hat A.
\end{align*}
We may thus deduce that $-A_1 - (\tilde\lambda - \alpha)\hat A\succeq 0$ whenever $\alpha \geq \delta\frac{4(1+2\zeta|\tilde\lambda|)}{\xi}$. Hence,
\begin{align*}
\tilde\lambda - \lambda \leq  \delta\frac{4(1+2\zeta|\tilde\lambda|)}{\xi}.
\end{align*}
Similarly,
\begin{align*}
(-A_1+B) - (\lambda - \alpha) (\hat A +C) &= -A_1 - \lambda \hat A + B - \lambda C +\alpha (\hat A+C)\\
&\succeq -\delta(1 + 2\zeta\abs{\lambda}) +\alpha(\hat A+C).
\end{align*}
We may thus deduce that $(-A_1+B) - (\lambda - \alpha) (\hat A +C)\succeq 0$ whenever $\alpha \geq \delta\frac{8(1+2\zeta\abs{\lambda})}{\xi}$. Hence,
\begin{align*}
\tilde\lambda - \lambda \geq  -\delta\frac{8(1+2\zeta\abs{\lambda})}{\xi}.
\end{align*}
Finally, we may estimate $\abs{\tfrac{1}{\lambda}}\geq \tfrac{\xi}{4}$ and $\abs{\tfrac{1}{\tilde\lambda}}\geq \tfrac{2\xi}{17}$. We conclude
\begin{align*}
-\delta\frac{8(1+8\zeta/\xi)}{\xi}&\leq \tilde \lambda - \lambda \leq \delta\frac{4(1+17\zeta/\xi)}{\xi}.\qedhere
\end{align*}
\end{proof}

\begin{proposition}
Let $\delta = \tfrac{\mu\xi^2}{360\zeta^3}$ and suppose $\tilde \lambda$ is the minimum eigenvalue of \eqref{eq:eigifp_perturbed}. Then,
\begin{align*}
\mu/2 \leq \mu(\tilde\gamma) \leq \tilde v^\top A(\tilde \gamma)\tilde v\leq \mu.
\end{align*}
\end{proposition}
\begin{proof}
The first inequality follows from \cref{lem:add_to_mult,lem:lambda_additive_bound}. The second inequality follows from the definition of $\mu$. The third inequality follows as
\begin{align*}
\tilde v^\top A(\tilde \gamma)\tilde v &= \tilde v^\top \left(\hat A + \tfrac{1}{\tilde\lambda} A_1 \right)\tilde v + 3\mu/4\\
&= \tilde v^\top \left((\hat A+C) + \tfrac{1}{\tilde\lambda} (A_1-B) - C + \tfrac{1}{\tilde\lambda}B \right)\tilde v + 3\mu/4\\
&\leq \norm{C} + \tfrac{1}{|\tilde\lambda|}\norm{B} + 3\mu/4\\
&\leq 4\delta\zeta  + 3\mu/4.
\end{align*}
Here, the first inequality holds as $(-A_1+B)\tilde v = \tilde\lambda(\hat A + C) \tilde v$. The second inequality follows as $\norm{C}\leq 2\delta\zeta$ and $\abs{\tilde\lambda} \geq \abs{\lambda} - \abs{\lambda - \tilde\lambda} \geq 1/2\zeta$.
\end{proof} 	\section{Numerical Experiment Tables}
\label{sec:tables}
We provide additional statistics for the numerical results plotted in \cref{fig:plot_1e3,fig:plot_1e4,fig:plot_1e5} for $n=10^3,\,10^4,\,10^5$, respectively. 
In \cref{table:1e3,table:1e4}, we present the averages for $n=10^3,\,10^4$ respectively over 100 random instances each, and in \cref{table:1e5} the averages for $n=10^5$ are given over 5 random instances. 
In these tables, \texttt{Error} and \texttt{ErrorCR} correspond to the error of $q_0(\tilde x)$ and the error of $\bar x $ within the convex reformulation respectively as defined in \cref{subsec:setup}.
For \texttt{WLK21}, \texttt{WK20} and \texttt{JL19}, we also report time for constructing the convex reformulation and solving the reformulation as \texttt{Ref.} and \texttt{Solve}. 
For each parameter combination, we highlight the algorithm with the smallest running time.

\begin{table}
\centering
{\ije{\footnotesize}{}
\scalebox{0.8} 
{
\begin{tabular}{ccccccc|ccccc}\toprule
& & \multicolumn{5}{c}{$\bar N = 10^4$} & \multicolumn{5}{c}{$\bar N = 10^5$} \\
\cmidrule(lr){3-7} \cmidrule(lr){8-12}
& & & & & \multicolumn{2}{c}{Time} & & & & \multicolumn{2}{c}{Time} \\
\cmidrule(lr){6-7} \cmidrule(lr){11-12}
$\bar\mu^*$   & Alg.         & \texttt{Error} & \texttt{ErrorCR} & Time      & Ref.\ & Solve & \texttt{Error} & \texttt{ErrorCR} & Time     & Ref.\  & Solve \\
\midrule
& \texttt{ WLK21 }&  4.8&  6.1& \textbf{ 0.1 }&  0.05&  0.05&  5.1&  5.4& \textbf{ 0.8 }&  0.3&  0.4 \\
& \texttt{ WK20 }&  5.7&  6.7&  0.5&  0.1&  0.3&  4.8&  5.3&  4.4&  0.5&  3.8 \\
1e-2& \texttt{ JL19 }&  1.5e+03&  1.8e+06&  0.7&  0.1&  0.6&  5.1e+01&  2.1e+06&  8.2&  0.6&  7.6 \\
& \texttt{ AN19 }&  6.7e+02& -&  1.5& -& -&  6.4e+02& -&  2.2& -& - \\
& \texttt{ BTH14 }&  4.2e+08& -&  1.1& -& -&  7.5e+08& -&  1.5& -& - \\
\midrule
& \texttt{ WLK21 }&  6.7&  7.2& \textbf{ 0.4 }&  0.2&  0.2&  8.5&  8.4&  2.9&  1.0&  1.8 \\
& \texttt{ WK20 }&  8.1&  7.1&  0.7&  0.1&  0.6&  7.0&  7.1&  7.0&  0.6&  6.4 \\
1e-4& \texttt{ JL19 }&  2.3e+09&  4.1e+12&  3.0&  0.1&  2.8&  1.0e+09&  3.6e+12&  49.9&  0.6&  49.3 \\
& \texttt{ AN19 }&  4.9& -&  1.6& -& -&  5.0& -&  2.4& -& - \\
& \texttt{ BTH14 }&  4.0e+08& -&  1.2& -& -&  4.4e+08& -& \textbf{ 1.7 }& -& - \\
\midrule
& \texttt{ WLK21 }&  6.5&  6.1& \textbf{ 0.8 }&  0.3&  0.5&  8.3&  8.2&  6.3&  1.8&  4.4 \\
& \texttt{ WK20 }&  6.4&  6.4&  1.6&  0.1&  1.5&  7.6&  8.2&  15.5&  0.5&  15.0 \\
1e-6& \texttt{ JL19 }&  7.9e+04&  7.5e+10&  3.1&  0.1&  3.0&  8.4e+04&  7.1e+10&  40.4&  0.5&  39.9 \\
& \texttt{ AN19 }&  1.4e+06& -&  1.7& -& -&  1.3e+06& -&  2.4& -& - \\
& \texttt{ BTH14 }&  1.3e+09& -&  1.4& -& -&  1.0e+09& -& \textbf{ 1.7 }& -& - \\
\bottomrule
\end{tabular}
}
 }
\caption{Average errors and solution times for $n=10^3$ over 100 random instances for each parameter combination. Note that errors are reported in units of $10^{-16}$.
We call attention to the setting $(\bar N,\bar\mu^*) = (10^5,10^{-6})$. In this setting, the fastest algorithm is \texttt{BTH14}. On the other hand, \texttt{BTH14} also reports the highest error of $\approx 10^{-7}$. \texttt{BTH14} is followed by \texttt{AN19} which achieves slightly smaller error of $\approx 10^{-10}$. While \texttt{WLK21} is slightly slower than both of these algorithms it achieves significantly smaller errors of $\approx 10^{-16}$. The results are similar for $(\bar N,\bar\mu^*) = (10^5,10^{-4})$ as well.}
\label{table:1e3}
\end{table}

\begin{table}
\centering
{\ije{\footnotesize}{}
\scalebox{0.8} 
{
\begin{tabular}{ccccccc|ccccc}\toprule
& & \multicolumn{5}{c}{$\bar N = 10^4$} & \multicolumn{5}{c}{$\bar N = 10^5$} \\
\cmidrule(lr){3-7} \cmidrule(lr){8-12}
& & & & & \multicolumn{2}{c}{Time} & & & & \multicolumn{2}{c}{Time} \\
\cmidrule(lr){6-7} \cmidrule(lr){11-12}
$\bar\mu^*$   & Alg.         & \texttt{Error} & \texttt{ErrorCR} & Time      & Ref.\ & Solve & \texttt{Error} & \texttt{ErrorCR} & Time     & Ref.\  & Solve \\
\midrule
& \texttt{ WLK21 }&  4.9&  6.4& \textbf{ 1.8 }&  0.8&  0.9&  4.7&  5.4& \textbf{ 11.1 }&  4.8&  4.8 \\
& \texttt{ WK20 }&  4.9&  5.7&  9.8&  1.6&  8.1&  5.3&  6.0&  67.5&  10.5&  56.8 \\
1e-2& \texttt{ JL19 }&  1.4e+02&  1.7e+06&  15.3&  1.6&  13.6&  6.3e+02&  1.8e+06&  93.8&  10.7&  82.8 \\
& \texttt{ AN19 }&  6.8e+02& -&  184.1& -& -&  1.2e+03& -&  324.5& -& - \\
\midrule
& \texttt{ WLK21 }&  1.5e+01&  1.6e+01& \textbf{ 6.6 }&  2.6&  3.7&  4.1e+01&  4.2e+01& \textbf{ 57.0 }&  24.0&  30.3 \\
& \texttt{ WK20 }&  1.0e+01&  1.1e+01&  16.6&  1.5&  15.1&  2.9e+01&  3.0e+01&  207.0&  11.0&  195.8 \\
1e-4& \texttt{ JL19 }&  6.7e+09&  4.2e+12&  57.9&  1.5&  56.4&  2.1e+10&  3.1e+12&  393.1&  11.3&  381.6 \\
& \texttt{ AN19 }&  4.3& -&  205.7& -& -&  4.5& -&  476.4& -& - \\
\midrule
& \texttt{ WLK21 }&  9.1e+01&  9.2e+01& \textbf{ 15.1 }&  5.1&  9.8&  2.7e+01&  2.8e+01& \textbf{ 130.7 }&  49.1&  79.0 \\
& \texttt{ WK20 }&  6.1e+01&  6.1e+01&  33.0&  1.5&  31.5&  3.1e+01&  3.1e+01&  264.0&  10.6&  253.2 \\
1e-6& \texttt{ JL19 }&  2.5e+09&  7.8e+10&  59.7&  1.5&  58.1&  1.6e+08&  7.1e+10&  402.7&  11.0&  391.4 \\
& \texttt{ AN19 }&  8.0e+06& -&  206.6& -& -&  4.4e+06& -&  475.5& -& - \\
\bottomrule
\end{tabular}
}
 }
\caption{Average errors and solution times for $n=10^4$ over 100 random instances for each parameter combination.
Note that errors are reported in units of $10^{-16}$.
}
\label{table:1e4}
\end{table}

\begin{table}
\centering
{\ije{\footnotesize}{}
\scalebox{0.8} 
{
\begin{tabular}{ccccccc|ccccc}\toprule
& & \multicolumn{5}{c}{$\bar N = 10^4$} & \multicolumn{5}{c}{$\bar N = 10^5$} \\
\cmidrule(lr){3-7} \cmidrule(lr){8-12}
& & & & & \multicolumn{2}{c}{Time} & & & & \multicolumn{2}{c}{Time} \\
\cmidrule(lr){6-7} \cmidrule(lr){11-12}
$\bar\mu^*$   & Alg.         & \texttt{Error} & \texttt{ErrorCR} & Time      & Ref.\ & Solve & \texttt{Error} & \texttt{ErrorCR} & Time     & Ref.\  & Solve \\
\midrule
& \texttt{ WLK21 }&  3.3&  9.9& \textbf{ 30.1 }&  12.6&  13.6&  5.3&  2.7& \textbf{ 229.2 }&  100.8&  101.7 \\
1e-2& \texttt{ WK20 }&  4.7&  7.8&  162.9&  24.7&  137.0&  3.1&  4.9&  1748.4&  527.9&  1216.3 \\
& \texttt{ JL19 }&  4.9&  1.4e+06&  287.3&  27.4&  259.1&  1.6e+02&  2.3e+06&  1930.7&  419.0&  1507.5 \\
\midrule
& \texttt{ WLK21 }&  1.5e+01&  1.6e+01& \textbf{ 141.6 }&  65.1&  70.8&  9.5e+01&  9.5e+01& \textbf{ 1586.3 }&  767.0&  728.5 \\
1e-4& \texttt{ WK20 }&  1.6e+01&  1.6e+01&  334.3&  25.7&  307.9&  1.4e+02&  1.4e+02&  10622.9&  437.7&  10180.8 \\
& \texttt{ JL19 }&  2.5e+09&  4.3e+12&  1044.3&  26.8&  1016.5&  9.2e+10&  8.7e+11&  11526.9&  514.5&  11007.9 \\
\midrule
& \texttt{ WLK21 }&  2.2e+01&  2.0e+01& \textbf{ 294.2 }&  97.8&  190.0&  6.2e+01&  6.4e+01& \textbf{ 3361.1 }&  1569.5&  1701.7 \\
1e-6& \texttt{ WK20 }&  1.5e+01&  1.6e+01&  612.3&  25.7&  585.6&  1.4e+02&  1.4e+02&  7781.5&  367.8&  7409.8 \\
& \texttt{ JL19 }&  7.6e+04&  8.5e+10&  1081.4&  19.5&  1061.2&  2.1e+06&  7.5e+10&  10960.0&  355.3&  10600.8 \\
\bottomrule
\end{tabular}
}
 }
\caption{Average errors and solution times for $n=10^5$ over 5 random instances for each parameter combination. Note that errors are reported in units of $10^{-16}$.
}
\label{table:1e5}
\end{table}

 \end{appendix}
\end{document}